\documentclass[11pt,reqno,oneside]{amsart}
 
\usepackage{svg}
\usepackage{esint}
\usepackage[T1]{fontenc} 
\usepackage{mathrsfs}  
\usepackage{amsmath}
\usepackage{amssymb} 
\usepackage{float}
\usepackage{amsfonts}
\usepackage{mathtools}
\usepackage{dutchcal}
\usepackage[utf8]{inputenc}
\usepackage[mathcal]{eucal} 
\usepackage{enumerate}
\usepackage{latexsym}
\usepackage{amsthm} 
\usepackage{hyperref}
\usepackage{xcolor}
\usepackage[totalwidth=14cm,totalheight=20cm, hmarginratio=1:1]{geometry}
\usepackage{csquotes}

\usepackage{tikz-cd}

\numberwithin{equation}{section}

\newtheorem{cor}{Corollary}[section]

\newtheorem{thm}[cor]{Theorem}

\newtheorem*{theorem*}{Theorem}

\newtheorem{prop}[cor]{Proposition}

\newtheorem{lemma}[cor]{Lemma}

\theoremstyle{definition}

\newtheorem{defn}[cor]{Definition}
\theoremstyle{remark}

\newtheorem{remark}[cor]{Remark}
\newtheorem*{rmk*}{Remark}

\usepackage{soul, xcolor, todonotes, easyReview}

\makeatletter
\newsavebox{\@brx}
\newcommand{\llangle}[1][]{\savebox{\@brx}{\(\m@th{#1\langle}\)}%
  \mathopen{\copy\@brx\kern-0.5\wd\@brx\usebox{\@brx}}}
\newcommand{\rrangle}[1][]{\savebox{\@brx}{\(\m@th{#1\rangle}\)}%
  \mathclose{\copy\@brx\kern-0.5\wd\@brx\usebox{\@brx}}}
\makeatother

\newcommand{\R}{\mathbb{R}} 

\newcommand{\sC}{\mathcal{C}}

\newcommand{\End}{\mathrm{End}}

\newcommand{\Stab}{\mathrm{Stab}}
\newcommand{\Ad}{\mathrm{Ad}}
\newcommand{\SL}{\mathrm{SL}}
\newcommand{\fsl}{\mathfrak{sl}}

\newcommand{\Ima}{\mathrm{Im}}

\newcommand{\SO}{\mathrm{SO}}

\newcommand{\C}{\mathbb{C}}

\newcommand{\hol}{\mathrm{hol}}
\newcommand{\dev}{\mathsf{dev}}

\newcommand{\Hol}{\mathrm{Hol}}

\newcommand{\Span}{\mathrm{Span}}
\newcommand{\sP}{\mathcal{P}}

\renewcommand{\sl}{\mathfrak{sl}(2,\C)}

\newcommand{\CP}{\mathbb{CP}}

\newcommand{\tec}{Teichm\"uller }

\newcommand{\sF}{\mathcal{F}}

\newcommand{\sH}{\mathcal{H}}

\newcommand{\sM}{\mathcal{M}}

\newcommand{\sR}{\mathcal{R}}

\newcommand{\sT}{\mathcal{T}}
\newcommand{\sV}{\mathcal{V}}
\newcommand{\ed}{\epsilon}
\newcommand{\oed}{\boldsymbol{\varepsilon}}
\newcommand{\roed}{\boldsymbol{\varepsilon}^*}
\newcommand{\mink}[1]{\langle #1\rangle}
\newcommand{\killi}[1]{\kappa(#1)}
\newcommand{\pair}[1]{\llangle #1\rrangle}
\newcommand{\cpack}{\mathfrak{C}}
\newcommand{\tr}{\mathrm{tr}}
\newcommand{\pol}{\mathpzc{q}}
\newcommand{\coc}{\mathtt{T}}
\newcommand{\vc}[1]{\mathsf{V}_{#1}}
\newcommand{\Yc}[1]{\mathscr{Y}_{#1}}
\newcommand{\lie}[1]{\mathtt{#1}}
\newcommand{\per}{\boldsymbol{\mu}}

\DeclareMathAlphabet{\mathpzc}{OT1}{pzc}{m}{it}

\DeclareMathOperator{\PSL}{\mathrm{PSL}}

\newcommand{\dS}{\mathrm{dS^3}}

\newcommand{\T}{\mathcal{T}}

\newcommand{\Link}{\mathrm{Link}}
\newcommand{\val}{\mathrm{val}}
\newcommand{\sTk}[1]{\sT^{(#1)}}
\newcommand{\sTor}{\sT^{(1)}_{\mathrm{or}}}
\newcommand{\tgmap}{\mathbf I}
\newcommand{\wh}{\mathfrak{w}}

\newcommand{\Pj}{\mathbf{P}}

\newcommand{\herm}{\mathsf{Herm}_2}
\newcommand{\genus}{\mathtt{g}}
\newcommand{\sel}{\mathsf{S}}

\newcommand{\Star}{\mathrm{star}}

\newcommand{\Zgroup}{Z^1_{\Ad\circ\rho}(\pi_1(\Sigma),\sl)}
\newcommand{\Bgroup}
{B^1_{\Ad\circ\rho}(\pi_1(\Sigma),\sl)}
\newcommand{\Hgroup}{H^1_{\Ad\circ\rho}(\pi_1(\Sigma),\sl)}
\title{Projective Rigidity of Circle Packings}
\author{Francesco Bonsante \and Michael Wolf }
\date{\today}
\begin{document}

\begin{abstract}
We prove that the space of circle packings consistent with a given triangulation on a surface of genus at least two is projectively rigid, so that a packing on a complex projective surface is not deformable within that complex projective structure.  More broadly, we show that the space of circle packings is a submanifold within the space of complex projective structures on that surface.
\end{abstract}

\maketitle
\tableofcontents

 \section{Introduction}

A still stunning theorem of Koebe \cite{Koebe:CirclePacking} asserts that each topological circle packing on a planar domain may be realized geometrically by round circles.  Rediscovered by Thurston \cite{Thurston:notes} who noted the relationship to work of Andreev \cite{andreev1, andreev2}, there are now a number of proofs from different perspectives in expanded settings. A partial list of relevant references around the Koebe-Andreev-Thurston Theorem and its extensions is   \cite{bobenko, chow-luo, ConnGort, rivin, schramm1, schramm2}. We also refer to \cite{steph, steph-book} for a nice and broad introduction to the topic. In this paper, we begin from the version that provides, on a closed surface of higher genus, a hyperbolic metric on the surface, say $\Sigma$, and a circle packing by round hyperbolic circles.

Kojima-Mizushima-Tan \cite{Kojima-Mizushima-Tan:Projective} (see also \cite{Kojima-Mizushima-Tan:Survey}, \cite{Kojima-Mizushima-Tan:Uniformization}), exploiting the invariance of the notion of round disk by complex projective transformations,  develop and make progress on a refined version of the question: given a conformal structure on such a surface and the combinatorics of a circle packing, is there a complex projective structure in that conformal structure for which there is a unique circle packing with the given combinatorics? In their work and this paper, a circle packing is a decomposition of the projective structure on a surface into round disks and complementary closures of curvilinear triangles. (Other authors allow a broader range of complementary domains.) The Kojima-Mizushima-Tan question (and conjecture in the affirmative) suggests an appealing picture: each combinatorial circle packing provides for a section of the space of complex projective structures over \tec space comprising structures on which the circle packing may be realized, with the section meeting the uniformized (or Fuchsian) locus at a single point.

That picture has at its foundation two questions, for each topological packing: can we deform a \enquote{round} circle packing within a fixed complex projective structure?  And, perhaps more crucially (see also \cite{Kojima-Mizushima-Tan:Projective}),  does the family of realizable complex projective structures admit a manifold structure?  The purpose of this paper is to settle these basic questions with answers that are optimistic for the full project: for a packing with triangular interstices on a surface of genus at least two, the set of projective structures that
admit round realizations of the packing is a manifold that meets each projective structure in at most a single point (representing at most a single round packing).

Thus our goal is to prove the following theorem.

\begin{thm}\label{thm:main}
    Let $\sT$ be a quasi-simplicial triangulation on $\Sigma$, a surface of genus $\genus(\Sigma)\geq 2$.  Then
    \begin{enumerate}[(i)]
        \item The moduli space $\sP_{\sT}$ of pairs of projective structures and circle packings with nerve $\sT$ admits a natural manifold
    structure of   of dimension $6\genus(\Sigma)-6$, so that
        \item The projection of $\sP_{\sT}$ to the space (of holonomies) of projective structures is a smooth immersion.
    \end{enumerate}
\end{thm}

Here the \enquote{natural} aspect of the manifold structure in (i) is determined by the condition in that the projection in (ii) should be an immersion (using the differentiable structure on the character variety of holonomies of projective structures).

The latter condition (ii) implies that the manifold $\sP_{\sT}$ admits a non-singular local parameterization in terms of the underlying projective structure, and so, for example, does not branch over the space of projective structures and also meets any particular projective structure in at most a singleton.

In general the combinatorics of a circle packing is described by a triangulation on the surface, that is called the \enquote{nerve} of the circle packing. That triangulation could admit loops or parallel edges, but its lift to the universal covering is always a simplicial triangulation. We call this class of triangulations on the surface \enquote{quasi-simplicial}; we recall the main features of this notion in Section \ref{ssec:quasi-simp}. 

The proof of Theorem~\ref{thm:main}(i) will be developed in Sections~\ref{sec:fixed nerve}, \ref{sec:vanishing}  and \ref{sec: manifold theorem}. Indeed, the proof of (i) will not rely on properties of (ii), but instead will realize the space $\sP_{\sT}$ as the regular values of a mapping on a larger space $\sM_{\sT}$. The proof of Theorem~\ref{thm:main}(ii), will be given in section~\ref{sec: projective rigidity}, using the vanishing theorem from Section~\ref{sec:vanishing}.

\subsection{Discussion of the literature.}

One of the main difficulties in studying circle packings on complex projective structures is the lack of a background metric in which disks can be defined as metric disks, thereby lacking a notion of center or radius.

As a result, the classical tools used in proving the Koebe-Andreev-Thurston theorem cannot be directly applied, and also the variational approaches developed in \cite{bobenko, deverdiere} do not seem to have a straightforward generalization in this context.

In their seminal work \cite{Kojima-Mizushima-Tan:Projective} Kojima, Mizushima, and Tan suggest a way to describe the space $\sP_{\sT}$. Their perspective differs from the approach presented in this paper. They begin with the observation that the moduli space of circle packings whose nerve is the union of two triangles sharing a common edge, is one-dimensional and is naturally parameterized by a quantity known as the cross ratio parameter.

Then, given a triangulation on a surface, the star of an edge consists of the two triangles incident to that edge. Thus, for a fixed circle packing on the surface, one can associate a cross ratio parameter to each edge, representing the corresponding piece of the packing in the star of that edge. This association yields a map:
\[
f:\sP_{\sT}\to\R^{e}
\]
where $e$ is the number of edges of the triangulation. They show that the map is injective and prove that the image is a real semialgebraic set of dimension $6\genus(\Sigma)-6$. However with their methods they could prove that the image is smooth only near solutions with quasi-Fuchsian holonomy -- which includes the Fuchsian solution featured in the Koebe-Andreev-Thurston result -- or for packings with a single disk.

In the last part of this paper, we prove that if we equip $\sP_{\sT}$ with the manifold structure provided by Theorem \ref{thm:main},
the map $f$ is a smooth embedding, proving that its image is a smooth submanifold in $\R^e$ and that our description is compatible with that given in \cite{Kojima-Mizushima-Tan:Projective}.

To the best of our knowledge, while the problem stated by Kojima, Mazushima, and Tan has attracted some interest, significant progress have been made only recently.  In \cite{dannenberg} the properness of the forgetful map $\pi:\sP_{\sT}\to\sT(\Sigma)$, which associates any projective structure equipped with a circle packing with its underlying conformal structure, has been proved under some technical assumptions.
More recently, Schlenker and Yarmola (\cite{schlenker-yarmola}) establish the general properness of the forgetful map. Furthermore, Lam (referenced as \cite{lam}) provides strong evidence supporting Kojima, Mizushima, and Tan's conjecture for surfaces of genus 1. In fact, Lam proves that in this case, the forgetful map $\pi:\sP_{\sT}\to\sT(\Sigma)$ is a branched covering with at most one branching point.

\subsection{Discussion of methods}

 We give a quick overview of the argument and then add a more detailed description. To begin, we can imagine a packing on a fixed complex projective structure and then a deformation of that packing.  The first difficulty one encounters, and the first clue to a possible solution, is that the vector fields that deform any given disk are restricted largely to that disk: of course disks can expand or shrink or translate, but most importantly, the points of tangency between neighboring disks can slide along the boundary circles-- those points of tangency of the circles are not generally projectively rigid.

This leads to a study of projective vector fields, one for each disk, with compatibility conditions between neighboring vector fields defining a global cohomological problem on the existence of a non-trivial cocycle, properly defined.
 
Ultimately, then, one wishes to prove a vanishing theorem.  For us, the proof of this theorem will eventually be the result of a combinatorial lemma that asserts that a decoration of the original triangulation by colorings and a partial orientation is trivial if the decoration is too restricted (see Proposition~\ref{prop:nobluevertex-bis}).

This combinatorial lemma bears some resemblance to Cauchy's Lemma for the rigidity of Euclidean polyhedra: see for instance Chapter $2$ of \cite{alexandrov}, or \cite{pak} for the more presently relevant infinitesimal approach to Dehn's \cite{Dehn16:rigidity} infinitesimal version. In particular, we adopt Pak's definitions and adapt his argument to negative Euler characteristic.
(Note that Pak cites \cite{Truskina81}, \cite{Schramm92:CageEgg} as inspirational.) 
Connelly-Gortler \cite{ConnGort}, as part of an approach to the Koebe-Andreev-Thurston theorem by rearranging and flowing packings, prove rigidity of circle packings on the sphere using an outline that is broadly analogous to the one we use; the details of their argument seem distinct from those in this paper.

\subsubsection{A more detailed development.}  
 We begin by describing some perspectives that inform how we frame the problem. First, a circle packing with curvilinear triangular complements defines a triangulation, say $\sT$, on the surface, where each disk defines a vertex and two vertices bound an edge when the corresponding disks are tangent. 
 
Next observe that the space of round disks in $\CP^1$ is a $\PSL(2,\C)$-homogenous space. It is classical fact that such a space is naturally identified to the Lorentzian de Sitter space $\dS$.

Let us then fix a triangulation with $v$ vertices, $e$ edges and $f$ triangles.
First we remark that giving a projective structure equipped with a circle packing with a fixed nerve is indeed equivalent to providing a map, called the  {\it selection map}, from the space $\widetilde{\sTk{0}}$ of vertices of the triangulation on the universal covering to the space of disks in $\CP^1$, that is equivariant by a non-elementary representation, called the {\it monodromy map}, into $\PSL(2,\C)$, such that disks corresponding to adjacent vertices are tangent.

Two selection maps determine the same surface with circle packing if and only if they differ by post-composition by elements in $\PSL(2,\C)$, so that their monodromies are conjugated. So we can realize the circle packing locus $\sP_\sT$ as a subset of the quotient $\sM$ of the space of equivariant maps of $\widetilde{\sTk{0}}$ into $\dS$ up to the group action.
Using that $\dS$ is naturally a manifold of dimension $3$, a standard computation shows that $\sM$ is a manifold of dimension $3v+12\genus(\Sigma)-12$, where $\genus(\Sigma)$ is the genus of the surface $\Sigma$. On the other hand, each tangency condition is expressible as a scalar equation in the Minkowski geometry of $\dS$, so $\sP_\sT$ is a locus described by e equations in a space of dimension $3v+12\genus(\Sigma)-12$. If we can prove that those equations are infinitesimally linearly independent near $\sP_\sT$, then one may conclude that $\sP_\sT$ is a manifold of dimension $3v+12\genus(\Sigma)-12-e=6\genus(\Sigma)-6$.

Thus the problem reduces to the computation of the dimension of the  kernel of the linearization of tangency equations.

\vskip.1cm

\noindent{\it A dimension computation.} 

As mentioned at the outset of the discussion, note that an infinitesimal deformation of a disk in $\CP^1$ can be represented by a projective vector field. Nonetheless, two projective vector fields encode the same infinitesimal deformation of a disk $\Delta$ if their difference is tangent to the stabilizer of $\Delta$ in $\PSL(2, \C)$. 
Given three mutually tangent disks, if we fix an infinitesimal deformation for each disk while preserving the tangency condition at the infinitesimal level, a unique projective vector field exists that encodes the infinitesimal variation of each disk. This fact directly stems from the rigidity of the configuration of three mutually tangent disks on $\CP^1$.

This observation enables us to describe the kernel of the linearization of tangency equations at a fixed selection map $\sel$ with monodromy $\rho$ as collections of maps from $\widetilde{\sTk{2}}$ into the space of projective vector fields. Those maps, denoted by $\lie{R}:\widetilde{\sTk{2}}\to\sl$, can be regarded as $\sl$-valued discrete $2$-forms on $\tilde\Sigma$, and must satisfy two conditions:
\begin{itemize}
\item the deformation of the selection map is through equivariant maps (here the monodromy of those maps is not assumed constant along the deformation), resulting in the condition that its codifferential $\lie{Q} = \delta\lie{R}$, a discrete one-form, is $\rho$-equivariant  under the action of the fundamental group. 

\item For a given vertex $\alpha$, if $\tau$ and $\tau'$ are triangles which share the vertex $\alpha$, then $\lie{R}(\tau)$ and $\lie{R}(\tau')$ encode the same infinitesimal deformation of the disk corresponding to $\sel(\alpha)$.
\end{itemize}
We prove that $\lie{R}$ satisfies the second condition if and only if $\lie{Q}=\delta\lie{R}$ lies in a certain totally real subspace of the subspace of the space of discrete $1$-forms
\[
  \sV^{(\rho,\sel)}=\{\lie{Q}\,|\, \vc{\lie{Q}(\oed)} \textrm{ has a double zero at } \mathsf{p}(\oed)\}\,,
\]
where $\mathsf{p}(\oed)$ represents the tangency point between the disks $\sel(\oed_-)$ and $\sel(\oed_+)$, and $\vc{\lie{Q}(\oed)}$ is the vector field on $\CP^1$ induced by $\lie{Q}(\oed)$.

Based on these ideas, we establish that the (real) dimension of the kernel of the linearization of tangency condition is indeed equal to the complex dimension of the subspace of the space of $\rho$-equivariant $\sl$-valued discrete $1$-forms on $\Sigma$ given by the intersection
\[
     \ker\delta\cap \sV^{(\rho,\sel)}\,.
\]
Now, the space of $\rho$-equivariant $\sl$-valued discrete $1$-forms on $\Sigma$ is naturally equipped with a non-degenerate complex symmetric form induced by the Killing form on $\sl$.  We have that elements of $\sV^{(\rho,\sel)}$ are isotropic. Moreover, a classical observation, based on a discrete integration by parts, is that the codifferential is the adjoint of the differential and so the orthogonal complement of $\ker\delta$ is the set $B^1_d$ of discrete $d$-exact one forms, i.e. those forms that can be globally defined as the discrete gradient of an $\sl$-valued function on vertices. On the other hand, a simple computation shows that
\[
  (\sV^{(\rho,\sel)})^\perp=\{\lie{Q}\,|\, \vc{\lie{Q}(\oed)} \textrm{ vanishes at } \mathsf{p}(\oed)\}\,,
\]
and some simple linear algebra implies that 
$\dim \ker\delta\cap \sV^{(\rho,\sel)}=6\genus-6$ if and only if $B^1_d$ and $(\sV^{(\rho,\sel)})^\perp$ are in direct sum.

To demonstrate the direct sum of these spaces, the vanishing of the intersection $B^1_d\cap(\sV^{(\rho,\sel)})^\perp$ is the heart of the paper. This result is also crucial to prove the projective rigidity as we will discuss later. For its proof we employ an argument similar to the one used in \cite{pak} 
to prove the rigidity of polyhedra in $\R^3$.
Specifically we have to prove that there is no non-trivial $\rho$-equivariant way to associate to each vertex of a triangulation on the universal covering a projective vector field, so that if two vertices are joined by an edge $\ed$, then the corresponding vector fields agree at $\mathsf{p}(\ed)$. Indeed the discrete gradient of such a $\rho$-equivariant $\sl$-valued discrete $0$-form would provide an element in $B^1_d\cap(\sV^{(\rho,\sel)})^\perp$.

Assuming that such a map $\lie{P}:\widetilde{\sTk{0}}\to\sl$,
which defines an infinitesimal projective motion for each vertex $\alpha$, exists, we construct a 
partial orientation of the $1$-skeleton of the 
surface in this way: given an  edge $\ed$ with endpoints $\alpha$ and $\beta$
we examine the orientation of the vector $\lie{P}(\alpha)(\mathsf{p}(\ed))=\lie{P}(\beta)(\mathsf{p}(\ed))$: if this vector points towards $\sel(\alpha)$ we orient $\ed$ towards $\alpha$, while if the vector points towards $\sel(\beta)$, we orient $\ed$ towards $\beta$. If the vector is tangent to the boundaries, we do not orient the corresponding edge. (A priori the orientation is constructed on the universal covering, but from the equivariance of $\lie{P}$ we see that such an orientation projects to a partial orientation of the $1$-skeleton of $\sT$.)

It turns out that such a partial orientation presents a {\it tight} behavior (see Definition \ref{defn: tight}). Although the condition in general is slightly technical, it has a simple description in the generic case where all the edges are oriented.
In that case for every vertex $\alpha$ there are at most two triangles with a vertex at $\alpha$ so that each triangle contains exactly one edge pointing towards $\alpha$.

On the other hand, since each oriented triangle $\tau$ must contain at least one vertex $\alpha$ such that exactly one oriented edge in the boundary of $\tau$ 
points  towards $\alpha$, a simple combinatorial argument, involving the Euler formula, shows that no tight partial orientation may exist on a surface.

This concise argument  works only under the genericity assumption that every edge is oriented, while the analysis of the general case requires additional technical details. Specifically we adapt the argument of \cite{pak} and \cite{Schramm92:CageEgg} to the case of higher genus surfaces, and  we conclude that such a partial orientation cannot exist.

Finally, we explain the proof of the second part of the theorem, namely the infinitesimal projective rigidity. This proof builds upon the same line of reasoning as before.

To this end, let us assume that there exists a deformation of the circle packing on a fixed projective surface. Each disk's infinitesimal deformation can be uniquely represented by a projective vector field normal to the disk boundary. By doing so, we obtain an $\rho$-equivariant discrete $0$-form $\lie{P}_0$.

By imposing the tangency condition, we find that the projective vector fields corresponding to the deformation of two tangent disks must coincide at the point of tangency. Using the previous notation, we basically have that $d\lie{P}_0\in B^1_d\cap(\sV^{(\rho,\sel)})^\perp$.
A revisiting of the vanishing result proved above shows then that  $\lie{P}_0=0$, so every disk is indeed infinitesimally fixed.

\subsection{Organization of the Paper.}
The first two sections, Section~\ref{sec:preliminaries} and Section~\ref{sec: circle packing preliminaries}, present the ingredients we will need for the arguments.  Section~\ref{sec:preliminaries} recalls the basics of complex projective geometry and vector fields, and then relates the space of disks to the de Sitter space $\dS$. We describe this space $\dS$ in terms of a space of Hermitian matrices of determinant $-1$, as well as the aspects of its geometry and infinitesimal isometries we will need. Of course, the paper imagines a circle packing as defining and defined by a triangulation, so we also recall some of the basics of such triangulations as well as cohomology theory based on such a triangulation with values in $\sl$. 

Section~\ref{sec: circle packing preliminaries} continues the background with a discussion of the local geometry of a circle packing, and then concludes by introducing the \enquote{selection} map with which we can analytically frame our problem.

Section~\ref{sec:fixed nerve} sets the background for the main result that the space of circle packings is a manifold: we define a manifold of images of circles together with a map $\tgmap$ whose zero locus $\tgmap^{-1}(0)$is our space of circle packings.  Our goal is to show that these points are regular for that mapping $\tgmap$.

Section \ref{sec:vanishing} contains the heart of the argument to prove the main Theorem \ref{thm:main}. We state the vanishing theorem, and we prove it using a combinatorial result. 

Section~\ref{sec: manifold theorem} is devoted to the proof of the main  Theorem~\ref{thm:main}{(i)} on the manifold structure.  We translate our regularity statement into a computation of the dimension of a space of cochains which we then recognize as the totally real subspace of a complex space whose dimension we can compute, using the vanishing theorem from the previous section.   

Finally in section~\ref{sec: projective rigidity}, we prove our projective rigidity theorem, again by expressing any projective deformation in terms of a non-trivial cochain, whose vanishing is then assured by our vanishing theorem.

\subsection{Acknowledgements.} We are grateful to Sadayoshi Kojima and Ser Peow Tan for bringing this problem to our attention. We benefited immensely from numerous conversations with Jean-Marc Schlenker and Wai Yeung Lam, the latter in particular for directing us towards \cite{pak}. 
The first author was partially supported by Blue Sky Research project “Analytic and geometric properties of low-dimensional manifolds”; he is a  member of the national research group GNSAGA.
The second author appreciates support from the NSF, specifically grants DMS-2005551 and DMS-1564374, as well as the GEAR network (NSF grants DMS-1107452, 1107263, 1107367 “RNMS: GEometric
structures And Representation varieties”) and the Simons Foundation.

\section{Preliminaries} \label{sec:preliminaries}
\subsection{Projective geometry and projective vector fields}\label{subsec:projective geometry}
For a non-zero vector $(z_0,z_1)\in\mathbb C^2$ we denote by $\mathsf{p}=[z_0:z_1]$ the projective line generated by $(z_0,z_1)$, considered as a point in $\CP^1$.
The standard affine chart is then given by the immersion 
$\mathbb C\to\CP^1$ sending $z$ to $[z:1]$. Its image is $\CP^1\setminus\{\infty\}$, where $\infty=[1:0]$.
So $z(\mathsf{p})=z_0/z_1$.
More generally an affine chart is given by the composition of  an affine immersion of $\C$ into $\C^2$ avoiding $0$ with the projection of $\C^2\setminus\{0\}\to\CP^1$. So in general an affine chart takes the form $w(\mathsf{p})=\frac{az_0+bz_1}{cz_0+dz_1}=\frac{a z(\mathsf{p})+b}{cz(\mathsf{p})+d}$ for some $a,b,c,d\in\C$ such that $ad-bc\neq 0$.
Notice that an affine chart covers the complement of a single point in $\CP^1$.

The group $\PSL(2,\C)$ acts on $\CP^1$ by projective transformations.
The action is clearly holomorphic and effective. So, it induces an inclusion of the Lie algebra $\sl$ into the algebra of vector fields on $\CP^1$. Given $\lie{A}\in\sl$
we will denote by $\vc{\lie{A}}$ the corresponding vector field on $\CP^1$ defined by
\begin{equation}\label{eq:projvect}
    \vc{\lie{A}}(\mathsf{p})=\frac{d\,}{dt}\Bigr|_{t=0}\,\exp(t\lie{A})\cdot \mathsf{p}\,.
\end{equation}
It is well-known that in this way $\sl$ is identified with the algebra of holomorphic vector fields on $\CP^1$. 
Moreover for any $\vc{\lie{A}}$ we have that its $z$-coordinate defined as 
\begin{align*}
    \pol_{\lie{A}}:\C&\to\C\\
    \pol_{\lie{A}}(z)&:=
    dz(\vc{\lie{A}}([z:1]))
\end{align*}
is a polynomial of degree $\leq 2$.
More precisely if $\lie{A}=\begin{pmatrix}\mathtt a & \mathtt b\\\mathtt c&-\mathtt a\end{pmatrix}\in\sl$, then 
$\exp(t\lie{A})=\begin{pmatrix}1+t\mathtt a & t\mathtt b\\t\mathtt c & 1-t\mathtt a\end{pmatrix}+o(t)$, so for a fixed $ z_0\in\C$ we have 
$\mathsf{p}(t)=\exp(t\lie{A})\cdot[z_0:1]=[(1+\mathtt at) z_0+\mathtt bt:\mathtt ctz_0+(1-\mathtt at)]+o(t)
$ and 
\[
   z(\mathsf{p}(t))=\frac{(1+t\mathtt a)z_0+t\mathtt b}{t\mathtt cz_0+(1-t\mathtt a)}+o(t)=z_0+t(-\mathtt cz_0^2+2\mathtt az_0+\mathtt b)+o(t)
\]
so $\pol_{\lie{A}}(z_0)=\frac{dz(\mathsf{p}(t))}{dt}(0)=-\mathtt c z_0^2+2\mathtt az_0+\mathtt b$.
In this way the correspondence $\lie{A}\to\pol_{\lie{A}}$ yields  an isomorphism 
\[
\sl\to \C_{\leq 2}[z]\,.
\]
\begin{remark}
    In fact for any affine chart $w$ on $\CP^1$, the $w$-coordinate of $\vc{\lie{A}}$ is a polynomial of degree $\leq 2$ in the variable $w$.
\end{remark}

\begin{remark}\label{rk:doublezero}
 A vector field $\vc{\lie{A}}$ may have two distinct simple zeros or a unique double zero.  We have that $\pol_{\lie{A}}(z)$ is a polynomial of degree $2$ if and only if $\vc{\lie{A}}(\infty)\neq 0$; $\pol_{\lie{A}}(z)$ is a polynomial of degree $1$ if and only if $\infty$ is a simple zero of $\vc{\lie{A}}$; $\pol_{\lie{A}}(z)$ is constant if and only if $\vc{\lie{A}}$ has a double zero at $\infty$.
\end{remark}
 We consider the complex bilinear form on $\sl$
 \[
 \killi{\lie{A},\lie{B}}=\tr(\lie{AB})
 \]
which is clearly invariant by the adjoint action.
A key remark is that
\[
  \killi{\lie{A},\lie{A}}=\frac{1}{2}\mathrm{Dis}(\pol_{\lie{A}})
\]
where $\mathrm{Dis}(\pol_{\lie{A}})$ is the discriminant of the polynomial $\pol_{\lie{A}}$.
\begin{remark}\label{rk:eval}
    The usual polarization formula and the previous identity shows that more generally
    \begin{equation}\label{eq:polarization of discriminant}
    \killi{\lie{A},\lie{B}}=\pm\frac{1}{4}\left(\mathrm{Dis}(\pol_{\lie{A}}\pm\pol_{\lie{B}})-\mathrm{Dis}(\pol_{\lie{A}})-\mathrm{Dis}(\pol_{\lie{B}})\right)\
    \end{equation}
    
    A useful consequence is that if $\vc{\lie{A}}$ has a double zero at some point $\mathsf{p}$ then we have that $\killi{\lie{A},\lie{B}}=0$ if and only if $\vc{\lie{B}}$ vanishes at $\mathsf{p}$.
    Indeed by acting by some element in $\PSL(2,\C)$ we can assume that $\mathsf{p}=\infty$. In that case $\pol_{\lie{A}}(z)= \mathtt b$ for some $\mathtt b\in\C$, and from \eqref{eq:polarization of discriminant} one sees that $\killi{\lie{A},\lie{B}}=0$ if and only if $\pol_{\lie{B}}$ has degree $\leq 1$, that is if and only if $\vc{\lie{B}}$ vanishes at $\infty$.
\end{remark}

\subsection{The space of disks of the sphere} 
A disk in $\CP^1$ is an open region that in the 
standard affine chart is bounded by a circle or a 
straight line. 
\begin{remark}
With this definition, if $\Delta$ is a disk, its exterior, defined as $\Delta^{\mathsf{C}}:= \CP^1\setminus\overline{\Delta}$ is still a disk. 
\end{remark}

Notably if $\Delta$ is a disk, then in every affine chart its boundary will be a circle or a line. In other words, the group $\PSL(2,\C)$ permutes the disks of $\CP^1$.

In order to give a geometric explanation of this phenomenon, we will give a more intrinsic characterization of disks of $\CP^1$.

We denote by $\herm$ the space of $2\times 2$ Hermitian matrices.
It is a $4$-dimensional real vector space, on which we consider the (real) quadratic form $\mathsf{X}\mapsto -\det(\mathsf{X})$.
Its polarization, denoted here by $\mink{\cdot,\cdot}$ has Lorentzian signature $(3,1)$. In other words $(\herm, \mink{\cdot,\cdot})$ is a copy of Minkowski $4$-space.

In this setting, an element of $\herm$ is said to be \emph{spacelike, lightlike, timelike} if the value of the quadratic form at that point is positive, null, negative.

Notice that timelike vectors correspond to definite Hermitian matrices. Semidefinite hermitian matrices form the lightlike cone. 

A time orientation is chosen so that positive  semi-definite Hermitian matrices are considered future-directed.

There is a natural action   of $\SL(2,\C)$ over $\herm$ given by 
$B\cdot\mathsf X=(B^{-1})^{T}\mathsf{X}\overline{B^{-1}}$.
The action factorizes to an action of the projectivized group $\PSL(2,\C)$ which preserves the 
Lorentzian product $\mink{\cdot,\cdot}$. Indeed the action induces a 
representation $\Phi:\PSL(2,\C)\to \mathrm{O}(\mink{\cdot,\cdot})$ that realizes an 
isomorphism between $\PSL(2,\C)$ and the connected component of the 
identity of the orthogonal group of $\mink{\cdot,\cdot}$, denoted here 
$\mathrm{SO}^+(\mink{\cdot,\cdot})$.

The de Sitter space is identified with the space of Hermitian matrices with determinant $-1$:
\[
  \dS:=\{\mathsf{X}\in\herm|\mink{\mathsf{X},\mathsf{X}} =1\}=
  \{\mathsf{X}\in\herm|\det\mathsf{X}=-1\} \,.
\]

Notice that any $\mathsf{X}\in\dS$ defines an indefinite Hermitian form over $\C^2$. Let us consider, for $\mathsf{X} \in \dS$,

\[
    \Delta(\mathsf{X})=\{[z_0:z_1]\in\CP^1 | \sum\mathsf{X}_{ij}z_i\bar{z}_j<0\}\,.
\]
\begin{prop}\label{pr:ds and disks}
For  $B\in\PSL(2,\C)$ and $\mathsf{X}\in\dS$ we have  $\Delta(B\cdot \mathsf{X})=B(\Delta(\mathsf{X}))$.
Moreover, for any  $\mathsf{X}\in\dS$, the set $\Delta(\mathsf{X})$ is an open disk of $\CP^1$; conversely, for any disk $\Delta$ in $\CP^1$ there exists a unique $\mathsf{X}\in\dS$ such that $\Delta=\Delta(\mathsf{X})$.
\end{prop}
In other words Proposition \ref{pr:ds and disks} states that $\dS$ naturally parameterizes the space of disks of $\CP^1$, and that this parameterization is equivariant under the action of $\PSL(2,\C)$.

In fact, this parameterization is continuous: a sequence of closed disks $\overline{\Delta(\mathsf{X_n})}$ converges to the closed disk $\overline{\Delta(\mathsf{X})}$ in the Hausdorff topology if and only if $\mathsf{X_n}$ converges to $\mathsf{X}$ in $\dS$.

\begin{proof}
    The first part of the proposition follows directly from the definitions.

    Let us prove the second part.
    Since $\PSL(2,\C)$ acts transitively on both $\dS$ and the space of disks of $\CP^1$, it is sufficient to prove that for
    \[
        \mathsf{X}_0=\left(\begin{array}{ll}1 &0\\0&-1\end{array}\right)
    \]
 we have that    $\Delta(\mathsf{X}_0)$ is a disk, and $\Stab_{\PSL(2,\C)}(\mathsf{X}_0)=\Stab_{\PSL(2,\C)}(\Delta(\mathsf{X}_0))$.

Notice that by definition we have
\[
\Delta(\mathsf{X}_0)=\{[z:w]| |z|^2-|w|^2<0\}
\]
so  $\infty\notin\Delta(\mathsf{X}_0)$, and in the standard affine chart
$\C\subset\CP^1$  we have that $\Delta(\mathsf{X}_0)$ is the standard unit disk centered in $0$.

It is well-known that the stabilizer in $\PSL(2,\C)$ of $\Delta(\mathsf{X}_0)$ is the subgroup $\mathrm{PU}(1,1)$, but this subgroup coincides with the stabilizer of the matrix $\mathsf{X}_0$.
\end{proof}
\begin{remark}
 Let us note that $\Delta(-\mathsf{X})$ coincides with the exterior of $\Delta(\mathsf{X})$.
\end{remark}

Now we make explicit some relations between the geometry of the disks in $\CP^1$ and the geometry of $\herm$.

Given a point $\mathsf{p}=[z_0:z_1]\in\CP^1$, the set of semipositive Hermitian matrices whose kernel contains $(z_0,z_1)$ determines a future directed lightlike ray  $\ell_{\mathsf{p}}$. In this way we point out a correspondence $\CP^1\to\mathbb P(\herm)$ which identifies $\CP^1$ with the set of lightlike lines of $\herm$.

\begin{lemma}\label{lm: p lies in Delta(X)}
Let $\mathsf{X}\in \dS$. Then  $\mathsf{p}\in\partial\Delta(\mathsf{X})$ if and only if $\mathsf{X}\in\ell_{\mathsf{p}}^\perp$.
\end{lemma}
\begin{proof}
The group  $\PSL(2,\C)$ acts transitively both on pairs $(\mathsf{p},\mathsf{X})$ such that $\mathsf{p}\in\partial\Delta(\mathsf{X})$, and on  pairs $(\ell,\mathsf{X})$ with $\ell$ a lightlike line in $\dS$ and $\mathsf{X}\in\dS$ such that $\mathsf{X}\in\ell^\perp$.The  transitivity of the action on the former set is evident. For the transitivity on the latter set observe that 
$\mathsf{X}^\perp$ is isometric to $\R^{2,1}$ by an isometry that conjugates the action of $\mathrm{stab}(\mathsf{X})$ on $\mathsf{X}^\perp$ to the action of 
the action $\mathrm{SO}(2,1)$ 
on $\R^{2,1}$. The transitivity of the action of $\PSL(2,\C)$ on pairs $(\ell,\mathsf{X})$ then easily follows using the transitivity of the action of $\PSL(2,\C)$ on $\dS$ and the well-known fact that $\mathrm{SO}(2,1)$ transitively acts on the set of lightlike lines.

  So it is sufficient to prove that for $\mathsf{p}_0=[0:1]=0$  and $\mathsf{X}_0=\begin{pmatrix}0&i\\-i& 0\end{pmatrix}$ we have $\mathsf{X_0}\in\ell_{\mathsf{p_0}}^\perp$.
  Notice indeed that $\Delta(\mathsf{X}_0)=\{[z:w]| i(z\bar w-w\bar z)<0\}$, so $\mathsf{p}_0\in\partial\Delta(\mathsf{X}_0)$.
  On the other hand, we observe that we may write $\ell_{\mathsf{p_0}}=\Span\left(\mathsf{Y}_0:=\begin{pmatrix}1&0\\0&0\end{pmatrix}\right)$, and we then compute
  \[
\mink{\mathsf{X}_0,\mathsf{Y}_0}=-\frac{1}{2}
\left(\det(\mathsf{X}_0+\mathsf{Y}_0)-\det\mathsf{X_0}-\det(\mathsf{Y}_0)\right)=0
    \]
  
\end{proof}
Notice that $\ell_{\mathsf{p}}^\perp$ is a $3$-space containing $\ell_{\mathsf{p}}$. The restriction of $\mink{\cdot,\cdot}$ on $\ell_{\mathsf{p}}^\perp$ is semipositive.

\begin{defn}
   We say that two disks $\Delta_1$ and $\Delta_2$ are {\it tangent}, if they are disjoint, and their closures meet at exactly one point. 
\end{defn}

In order to understand which conditions $\mathsf{X}_1$ and $\mathsf{X}_2$ in $\dS$
have to satisfy in order to guarantee that $\Delta(\mathsf{X}_1)$ and $\Delta(\mathsf{X}_2)$ are tangent, we start with a simple remark.

\begin{lemma}\label{lm:ds tangent}
Let $\mathsf{X}$ and $\mathsf{X}'$ be points in $\dS$. Then $\mink{\mathsf{X},\mathsf{X}'}=-1$ if and only if $\mathsf{X}+\mathsf{X}'$ is lightlike (including the case $\mathsf{X}+\mathsf{X}=0$).
Moreover the action of $\PSL(2,\C)$ on
\[
\mathcal B=\{(\mathsf{X},\mathsf{X}')\in(\dS)^2|\mink{\mathsf{X},\mathsf{X'}}=-1\}\,.
\]
contains exactly $3$ orbits:
\begin{align*}
    \mathcal B_+=&\{(\mathsf{X},\mathsf{X}')\in(\dS)^2|\mathsf{X}+\mathsf{X'} \text{ is future directed lightlike vector}\}\,,\\
    \mathcal B_0=&\{(\mathsf{X},\mathsf{X}')\in(\dS)^2|\mathsf{X}=-\mathsf{X'}\}\,,\\
    \mathcal B_-=&\{(\mathsf{X},\mathsf{X}')\in(\dS)^2|\mathsf{X}+\mathsf{X'} \text{ is past directed lightlike vector}\}\,.
\end{align*}
\end{lemma}
\begin{proof}
    The first part is a direct computation.
    Indeed we see that the map
    \[
    (\mathsf{X},\mathsf{X'})\mapsto(\mathsf{X},\mathsf{X+X'})
    \]
    realizes a $\PSL(2,\C)$-equivariant map between $\mathcal B$ and the bundle on $\dS$ of tangent lightlike vectors  given by
    \[
     T^0\dS=\{(\mathsf{X},\mathsf{Y})\in\dS\times\herm|\mink{\mathsf{X},\mathsf{Y}}=0 \text{ and }\mink{\mathsf{Y},\mathsf{Y}}=0\}
    \]
    It is well-known that $\PSL(2,\C)\cong\SO^+(\mink{\cdot,\cdot})\cong\SO^+(3,1)$ acts transitively on future directed lightlike tangent vectors (resp. on past directed lightlike vectors), so the statement easily follows.

\end{proof}

\begin{lemma}\label{lm:tangent point}
Two disks $\Delta(\mathsf{X})$ and $\Delta(\mathsf{X'})$ are tangent if and only if $\mathsf{X}+\mathsf{X'}$ is a future directed lightlike vector\,. 

Moreover in this case  $\mathsf{p}=\ker(\mathsf{X}+\mathsf{X}')$ is the tangency point between the closure of disks $\Delta(\mathsf X)$ and $\Delta(\mathsf X')$.
\end{lemma}   
\begin{proof}
       Consider the family 
\[
\mathcal A=\{(\mathsf{X}, \mathsf{X'})\in(\dS)^2| \Delta(\mathsf{X})\textrm{ and }\Delta(\mathsf{X}')\textrm{ are tangent}\}\,.
\]
This family is invariant under the product action of $\PSL(2,\C)$, and the induced action is transitive.
We conclude that the scalar product  $\mink{\mathsf{X},\mathsf{X'}}$ is constant for $(\mathsf{X},\mathsf{X}')\in\mathcal A$.
In order to prove that the constant is $-1$ it is sufficient to notice  that
for any $\mathsf{X}\in\dS$ the pair $(\mathsf{X}, -\mathsf{X})\in\overline{\mathcal A}$. This follows from the simple observation that the exterior of $\Delta(\mathsf{X})$ can be approximated by disks tangent to $\Delta(\mathsf{X})$ and we have already noticed that the exterior of $\Delta(\mathsf{X})$ is given by $\Delta(\mathsf{-X})$.

In this way we have proved that if $\Delta(\mathsf{X})$ and $\Delta(\mathsf{X}')$ are tangent, then $\mink{\mathsf{X},  \mathsf{X}'}=-1$, and so $\mathsf{X}+\mathsf{X'}$ is lightlike. 
Notice that $\mathsf{X'}\neq-\mathsf{X}$ since $\Delta(\mathsf{X}')$ is not the exterior of $\Delta(\mathsf{X})$.

We conclude that $\mathsf X+\mathsf X'$ is either semipositive or seminegative. On the other hand $\Delta(\mathsf X)\cup\Delta(\mathsf X')$ does not cover $\CP^1$. On a line of $\C^2$ corresponding to a point $\mathsf{q}$ in the complement of $\overline{\Delta(\mathsf X)\cup\Delta(\mathsf X')}$, both $\mathsf X$ and $\mathsf X'$ are positive. So the form corresponding to $\mathsf{X}+\mathsf{X}'$ cannot be seminegative.
We deduce that $\mathsf{X}+\mathsf{X}'$ is future directed.

So we have proved that $\mathcal A\subset\mathcal B_ +$. In order to conclude that these sets are equal it is sufficient to notice that they are both invariant under the action of $\PSL(2,\C)$, and that actually $\mathcal B_+$ is an orbit for such an action.

\end{proof}

Given $\lie{A}\in\sl$ we denote by $\Yc{\lie{A}}\in \mathfrak{so}(\mink{\cdot,\cdot})$ the linear transformation 
of $\herm$ given by
\[\Yc{\lie{A}}(\mathsf{X})=\frac{d\,}{dt}\Bigr|_{t=0}\,\left(\Phi(\exp(t\lie{A}))(\mathsf{X})\right)\,,
\]
where $\Phi:\PSL(2,\C)\to \mathrm{O}(\mink{\cdot,\cdot})$ was the
isomorphism defined earlier.

We want to relate the transformation $\Yc{\lie{A}}$ with the geometry of the vector field $\vc{\lie{A}}$ on $\CP^1$.

\begin{lemma}\label{lm: stabilizer of a disk}
    For $\lie{A}\in\sl$ and $\mathsf{X}\in\dS$ we have that $\Yc{\lie{A}}(\mathsf{X})=0$ if and only if $\vc{\lie{A}}$ is tangent to $\partial\Delta(\mathsf{X})$\,.
\end{lemma}
\begin{proof}
    The stabilizer in $\PSL(2,\C)$ of $\mathsf{X}$  coincides with the stabilizer of the disk $\Delta(\mathsf{X})$.
    Now notice that for $\lie{A}\in\sl$ the group $\exp(t\lie{A})$ stabilizes $\mathsf{X}$ if and only if $\Yc{\lie{A}}(\mathsf{X})=0$, whereas the same group stabilizes $\Delta(\mathsf{X})$ if and only if the vector field $\vc{\lie{A}}$ is tangent to $\partial\Delta(\mathsf{X})$. The conclusion follows.
\end{proof}

The key observation we will use in the paper is the following:

\begin{lemma}\label{lm:doublezero}
    If $\Delta(\mathsf{X})$ and $\Delta(\mathsf{X'})$ are tangent disks and $\vc{\lie{A}}$ is tangent to both $\partial\Delta(\mathsf{X})$ and $\partial\Delta(\mathsf{X}')$, then $\vc{\lie{A}}$ has a double zero at the tangency point $\mathsf{p}$.
\end{lemma}

\begin{proof}
    We can fix an affine coordinate $z$ so that $\Delta(\mathsf{X})=\{z\in\C\,|\,\Im(z)>0\}$ while
    $\Delta(\mathsf{X}')=\{z\in\C\,|\,\Im(z)<-1\}$ so that the tangency point $\mathsf{p}=\infty$.
    We note that that $\vc{\lie{A}}$ is tangent to $\partial\Delta(\mathsf{X})$ if and only if $\pol_{\lie{A}}(z)=\mathtt{a}z^2+\mathtt{b}z+\mathtt{c}$, with $\mathtt{a},\mathtt{b},\mathtt{c}$ real.
    Similarly, $\vc{\lie{A}}$ is tangent to $\partial\Delta(\mathsf{X'})$ if and only if $\pol_{\lie{A}}(z)=\mathtt{a}'(z+i)^2+\mathtt{b}'(z+i)+\mathtt{c}'$, with $\mathtt{a}',\mathtt{b}',\mathtt{c}'$ real.

    Now imposing that the two expressions for $\pol_{\lie{A}}$ agree, we see that $\mathtt{a}z^2+\mathtt{b}z+\mathtt{c}=\mathtt{a}'(z+i)^2+\mathtt{b}'(z+i)+\mathtt{c}'$, and we deduce that
    \begin{align*}
    \mathtt{a}&=\mathtt{a}'\,,\\
\mathtt{b}&=\mathtt{b}'+2i\mathtt{a}'\,,\\
\mathtt{c}& =\mathtt{c}'-\mathtt{a}'+i\mathtt{b}'.
    \end{align*}
Now using that all of the coefficients of the polynomials are real, we deduce that
$\mathtt{a}=\mathtt{a}'=\mathtt{b}=\mathtt{b}'=0$ and using Remark \ref{rk:doublezero} we obtain the result.
\end{proof}

\subsection{Triangulation of a surface}\label{ssec:quasi-simp}
Let $\Sigma$ be a surface possibly with boundary.
A triangulation $\sT$ of a surface $\Sigma$  is a triple $(\sTk{0}, \sTk{1}, \sTk{2})$, where
\begin{itemize}
\item $\sTk{0}$ is a discrete set of $\Sigma$, whose elements are called vertices;
\item $\sTk{1}$ is a set of embedded arcs or loops between vertices;
\item $\sTk{2}$ is the set of complementary regions, called faces,
of the union of the edges, which we call the $1$-skeleton of the triangulation $\sT$ and denote by $\Sigma^{(1)}$.
\end{itemize}
We moreover require:
\begin{itemize}
\item Distinct edges may meet only at their endpoints;
\item $\Sigma^{(1)}$ contains the boundary of $\Sigma$;
\item The number of edges incident to a given vertex $\alpha$, called here the {\it valence} of the vertex and denoted by $\val(\alpha)$, is at least $3$;
\item Faces are topologically  disks whose boundary comprises the union of exactly $3$ edges. 
\end{itemize}
We consider triangulations as combinatorial data, so they are defined up to isotopy of the surface.

We allow edges to be loops or to be parallel; here as usual, edges are parallel if they share endpoints.  So in general $\Sigma^{(1)}$ is a multi-graph.
In particular, in general a triangulation 
does not induce a simplicial structure on $\Sigma$.

We remark that if an edge forms a loop with vertex at $\alpha$, it counts twice in the computation of $\val(\alpha)$.

Each edge admits two orientations. An oriented edge is an edge equipped with a orientation. We denote by $\sTor$ the set of 
oriented edges so that the natural forgetful map $\sTor\to\sTk{1}$ is 
 $2:1$. 
 We usually denote by $\oed$ an oriented edge and by $\ed$ its 
unoriented support. Given an oriented edge $\oed$ we denote by $\roed$
the edge obtained by reversing the orientation.
Each oriented edge has a initial point $\oed_-$ and a terminus $\oed_+$, 
which are clearly vertices.

\begin{remark}\label{rk:order}
Given a vertex $\alpha$ we denote by $\sTor(\alpha)$ the set of oriented edges with initial point at $\alpha$. If $\alpha$ is an interior vertex, the orientation of $\Sigma$ induces a natural cyclic permutation on $\sTor(\alpha)$ that we denote by $\sigma_\alpha$.
On the other hand, if $\alpha$ is a boundary vertex, then the orientation of $\Sigma$ induces a total order is defined on $\sTor(\alpha)$.
\end{remark}

Given a triangulation $\sT$ we denote by $v$ the number of vertices, by $e$ the number of edges, and by $f$ the number of faces. We always have
\[ v-e+f=\chi(\Sigma)\,.\]
When $\Sigma$ has boundary, we distinguish  boundary vertices and edges, and interior edges and vertices.
We denote by $v_b, e_b$ the number of boundary vertices and edges, and by $v_i, e_i$ the number of interior vertices and edges. 
Since the boundary is a disjoint union of circles, it has zero Euler characteristic so we always have $v_b=e_b$.

We have the following simple combinatorial relations.
\begin{lemma}\label{lm:comb-triang}
Let $\sT$ be a triangulation on a surface $\Sigma$ possibly with boundary. Let $k=v_b=e_b$. Then
\begin{itemize}
\item $v_i-e_i+f=\chi(\Sigma)$;
\item $3f=2e_i+k=\sum_{\alpha\in\sTk{0}}\val(\alpha)-k$;
\item If $\Sigma$ is closed, then $e-3v=-3\chi(\Sigma)$.s
\end{itemize}
\end{lemma}
\begin{proof}
The first identity follows from the fact that $v-e=(v_i+k)-(e_i+k) = v_i-e_i$.

For the second identity we notice that the each face contains exactly 
$3$ edges. On the other hand, each interior edge is  in the boundary of two distinct faces, while every boundary edge is in the boundary of one face.
So the number of pairs $(\ed,\tau)$, where $\ed$ is an edge contained in the boundary of a face $\tau$, is equal to $3f$ but also to $2e_i+k$, and we deduce that $3f=2e_i+k=2e-k$.

On the other hand the number of oriented edges is equal to $2e$ but also to $\sum_{\alpha\in\sTk{0}}\val(\alpha)$. So 
$3f=2e-k=\sum_{\alpha\in\sTk{0}}\val(\alpha)-k$.

The third identity is obtained by combining the first two.
\end{proof}

Recall that the \emph{star} of a vertex $\alpha$, denoted here by $\Star(\alpha)$, is the union of all the closed faces which contain $\alpha$. 
We say that a triangulation is {\it simplicial} when the star of each vertex is 
an embedded closed disk. It is easy to see that this happens if and 
only if $\Sigma^{(1)}$ is a graph: it does not contain loops or parallel edges.
In particular, each edge is determined by its endpoints.

While we do not require  triangulations to be simplicial, triangulations occuring in this work will belong to an intermediate class of triangulations.

\begin{defn}\label{defn:quasi simplical}
A triangulation $\sT$ is {\it quasi simplicial} 
if there exists a finite covering $\hat\Sigma\to\Sigma$ such that the lifting of $\sT$ to $\hat\Sigma$ is simplicial.
\end{defn}

Given a triangulation $\sT$, a monogon is a \emph{simply connected} region of $\Sigma$ bounded by a closed edge. A bigon is a simply connected region of $\Sigma$ bounded by two parallel edges.

\begin{lemma}\label{lm:quasi-simplicial}
Given a triangulation $\sT$ of $\Sigma$ the following are equivalent:
\begin{enumerate}
\item $\sT$ is quasi simplicial;
\item $\sT$ contains neither monogons nor bigons;
\item the lifting of $\sT$ to the universal covering $\tilde\Sigma$ is a simplicial triangulation.
\end{enumerate}
\end{lemma}
\begin{proof}
 Clearly if $\sT$ contains a monogon or a bigon, so does any covering of $\sT$. So it cannot be quasi-simplicial. This proves $(1)\Rightarrow (2)$.
 
 On the other hand assume that $\Sigma$ contains neither bigons nor 
 monogons, and let us prove that no edge on $\tilde\Sigma$ is closed.
 If $\tilde\ed$ were closed on $\tilde\Sigma$, then the corresponding 
 edge $\ed$ on $\Sigma$ would be homotopically trivial, but then it 
 would bound a monogon. Similarly two edges in $\tilde\Sigma$ cannot 
 form a 
 loop, otherwise its projection would bound a bigon in $\Sigma$.
 But this implies that the triangulation $\tilde{\sT}$ is 
 simplicial. This proves the implication $(2)\Rightarrow (3)$.

Finally let us assume that $\tilde\sT$ is simplicial. 
Let $c_1,\ldots,c_N$ be the simple loops contained in $\Sigma^{(1)}$ made either by a single edge or by two edges. None of them is homotopically trivial.
The lifting of $\sT$ to a covering $\hat\Sigma$ is a simplicial triangulation if and only if there is no closed lift of the curves $c_i$ in $\hat\Sigma$. This happens exactly when no element in $\pi_1(\hat\Sigma,\hat{p}_0)$ -- which is identified to a subgroup of $\pi_1(\Sigma,p_0)$ -- is freely homotopic to one among $c_1,\ldots, c_N$. Now let us fix elements $g_i\in\pi_1(\Sigma, p_0)$ freely homotopic to $c_i$, and  recall that an element $g\in\pi_1(\Sigma, p_0)$ is freely homotopic to $c_i$ if and only if it is conjugate to $g_i$.  It follows that any normal covering, corresponding to a finite index normal subgroup which does not contain $g_1,\ldots, g_N$, satisfies the condition above. 
Since $g_i\neq 1$, and $\pi_1(\Sigma, p_0)$ is residually finite \cite{Baumslag}, such a subgroup exists.
\end{proof}

When $\sT$ is a simplicial triangulation, the link of a vertex $\alpha$, denoted by $\Link(\alpha)$, is defined as the boundary of the star of $\alpha$, that is the union of edges $\ed$ so that $\alpha$ is not an endpoint of $\ed$ and  there is a triangle $\tau$ containing $\alpha$ and $\ed$.

If $\alpha$ is an interior vertex, then $\Link(\alpha)$  is a circuit contained in the $1$-skeleton of $\sT$. It is naturally oriented as the boundary of $\Star(\alpha)$. So a cyclic permutation is defined on the set of vertices  $\Link_0(\alpha)=\Link(\alpha)\cap\sTk{0}$.
The map $\sTor(\alpha)\to\Link_0(\alpha)$ sending $\oed$ to $\oed_+$
is an bijection equivariant by the action of the cyclic group $\mathbb Z/\val(\alpha)\mathbb Z$.

If $\alpha$ is a boundary edge, then $\Link(\alpha)$ is an oriented interval. So a total order is defined on the set $\Link_0(\alpha)$. 
Also in this case, recalling the order defined in Remark \ref{rk:order}  on $\sTor(\alpha)$, we see that the map $\sTor(\alpha)\to\Link_0(\alpha)$ defined as above is an order-preserving bijection. 

\subsection{The relevant chain complexes.} \label{ssec: chain complexes}
In this section we will recall some well-known facts about the cohomology of the fundamental group of a triangulated surface with coefficients in a $\pi_1(\Sigma)$-module.
We refer to \cite{Labourie-survey} for a detailed discussion.

We fix a surface $\Sigma$ without boundary and a triangulation $\sT$ of $\Sigma$. Clearly $\pi_1(\Sigma)$ acts on the lift $\widetilde{\sT}$ of $\sT$ to the universal covering $\tilde\Sigma$.
Such an action restricts to an action on the space of vertices $\widetilde{\sTk{0}}$, on the space of faces $\widetilde{\sTk{2}}$ and on the space of \emph{oriented} edges $\widetilde{\sTor}$.

Fix now a representation $\rho:\pi_1(\Sigma)\to\PSL(2,\C)$. We consider three spaces. 
 \begin{align*}
     \sC^0_{\rho} = \{\lie{P}: \widetilde{\sTk{0}} \to \sl:&\, \lie{P} \text{ is } \Ad\circ\rho \text{-equivariant }\}\\
      \sC^1_{\rho} = \{\lie{Q}: \widetilde{\sTor} \to \sl:&\, \lie{Q} \text{ is } \Ad\circ\rho \text{-equivariant }, \lie{Q}(\roed)= - \lie{Q}(\oed)\}\\
       \sC^2_{\rho} = \{\lie{R}: \widetilde{\sTk{2}} \to \sl:&\, \lie{R} \text{ is } \Ad\circ\rho \text{-equivariant }\}\,.
 \end{align*}
 Given a triangle $\tau$, its boundary edges are naturally oriented by $\tau$.
 We will briefly say that an \emph{oriented} edge $\oed$ lies in the boundary of $\tau$, and we write $\oed\subset\partial\tau$,  when both $\oed$ is contained in $\partial\tau$ \emph{and} its orientation agrees with the orientation induced by $\tau$.

 \begin{figure}[htbp]
  \centering
  \includegraphics[width=9cm]{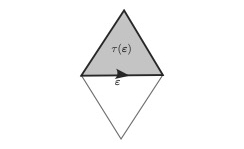}
  \caption{In the picture it is shown the unique triangle containing an oriented edge, according to the convention.}\label{fig-oriented-edges}
\end{figure}

Notice that with the convention any oriented edge $\oed$ is included in a unique triangle, which will be denoted by $\tau(\oed)$, see Figure \ref{fig-oriented-edges}.

From these spaces of cochains, we define two complexes:
\begin{align*}
    \sC^0_{\rho}  &\overset{d}{\rightarrow}  \sC^1_{\rho}  \overset{d}{\rightarrow}   \sC^2_{\rho}\\
      \sC^0_{\rho}  &\overset{\delta}{\leftarrow}  \sC^1_{\rho}  \overset{\delta}{\leftarrow}   \sC^2_{\rho}
\end{align*}
with the following defining conditions. 
\begin{align*}
    d\lie{P}(\oed) &=\lie{P}(\oed_+) - \lie{P}(\oed_-) & \textrm{for }\oed\in\widetilde{\sTor};\\
    d\lie{Q}(\tau) &= \sum_{\oed\subset\partial\tau}\lie{Q}(\oed) & \textrm{for }\tau\in\widetilde{\sTk{2}};\\
    (\delta \lie{R})(\oed) &= \lie{R}(\tau(\oed)) - \lie{R}(\tau(\roed)) & \textrm{for }\oed\in \widetilde{\sTor}; 
    \\
    (\delta \lie{Q})(\alpha) &= \sum_{\oed\,:\, \oed_-=\alpha} \lie{Q}(\oed)& \textrm{for }\alpha\in\widetilde{\sTk{0}}.
    \end{align*}

At the same time, we notice that the invariant form $\killi{\cdot,\cdot}$ on $\sl$ induces  natural non-degenerate pairings $\sC^0_{\rho} \times \sC^0_{\rho} \to \C$, $\sC^1_{\rho} \times \sC^1_{\rho} \to \C$, and $\sC^2_{\rho} \times \sC^2_{\rho} \to \C$ in this way.
First notice that for $\lie{P},\lie{P}'\in\sC^0_{\rho}$ and $\alpha\in\sTk{0}$, the product $\killi{\lie{P}(\tilde\alpha}, \lie{P}'(\tilde\alpha))$ is independent of the lifting $\tilde\alpha\in\widetilde{\sTk{0}}$.
So the following definition makes sense:
\[
\pair{\lie{P},\lie{P}'}=\sum_{\alpha\in\sTk{0}}\killi{\lie{P}(\tilde\alpha), \lie{P}'(\tilde\alpha')}\,.
\]
Similarly we can define
\begin{align*}
      \pair{\lie{Q}, \lie{Q}'} &= \sum_{\oed \in \sTor}\killi{\lie{Q}(\widetilde{\oed}),\lie{Q}'(\widetilde{\oed})}\,,\\
        \pair{\lie{R}, \lie{R}'} &= \sum_{\tau \in \sTk{2}}\killi{\lie{R}(\tilde{\tau}),\lie{R}'(\tilde{\tau})},
\end{align*}
where here of course, $\widetilde{\oed}$
and $\tilde{\tau}$ denote respectively any lifting of $\oed$ (as anoriented edge) and of $\tau$.

\begin{prop}\label{pr:adjoint}
    The operator $\delta$ is the adjoint of the operator of $d$, i.e
    \begin{align*}
        \pair{d\lie{Q}, \lie{R}} &= \pair{\lie{Q}, \delta \lie{R}}\\
        \pair{d\lie{P}, \lie{Q}} &= -\pair{\lie{P},\delta \lie{Q}}\,.
        \end{align*}
\end{prop}

For the proof see Proposition 4.2.9 of \cite{Labourie-survey}.

Certainly, we would like to compute the quotient $H^1(\sC^*_{\rho}, d)=\ker[d_1: \sC^1_{\rho}\to \sC^2_{\rho}]/\Ima[d_0:\sC^0_{\rho} \to \sC^1_{\rho}] = Z^1(\sC_\rho, d)/B^1(\sC_\rho, d)$, where as usual we denote by $Z^1(\sC^*_\rho, d)$ the kernel of the  operator $d_1: \sC^1_{\rho}\to \sC^2_{\rho}$, and by $B^1(\sC^*_\rho, d)$ the image of $d_0:\sC^0_{\rho} \to \sC^1_{\rho}$.
Let us remark that the complexes $(\sC^{*}_\rho, d)$ and $(\sC^{*}_\rho, \delta)$ can be regarded as a subcomplexes of  bigger complexes defined on the space $\tilde{\sC}^*:=(\sl)^{\widetilde{\sTk{*}}}$ of all maps -- in particular not necessarily $\rho$-equivariant -- from $\widetilde{\sTk{*}}$ to $\sl$.
Both the complexes $(\tilde{\sC}^*,d)$ and  $(\tilde{\sC}^*,\delta)$ compute the cohomology of $\tilde\Sigma$ with coefficients in $\sl$ \cite{Labourie-survey}. So those complexes are acyclic. 

It follows that if $\lie{Q}\in\sC^{1}_\rho$ is a $d$-cocycle, there exists $\lie{P}\in\tilde{\sC}^{0}$ such that $\lie{Q}=d\lie{P}$.
However in general $\lie{P}:\widetilde{\sTk{0}}\to\sl$ is not $\rho$-equivariant. The equivariance of $\lie{Q}$ implies that, for any $g\in\pi_1(\Sigma)$, the $0$-cochain 
$\alpha\mapsto \lie{P}(g\alpha)-\mathrm{Ad}(\rho(g)) \lie{P}(\alpha)$ is $d$-closed, and consequently, constant. So we find a map $\coc_{\lie{P}}:\pi_1(\Sigma)\to\sl$
such that
\begin{equation}\label{eq:primitive equivariance}
    \lie{P}(g\alpha)=\mathrm{Ad}(\rho(g)) \lie{P}(\alpha)+\coc_{\lie{P}}(g)
\end{equation}
for all $\alpha\in\widetilde{\sTk{0}}$ and for all $g\in\pi_1(\Sigma)$.

Using \eqref{eq:primitive equivariance}
we see that $\coc_{\lie{P}}$ satisfies the cocycle relation
\begin{equation}\label{eq:coc}
\coc_{\lie{P}}(gh)=\coc_{\lie{P}}(g)+\mathrm{Ad}(\rho(g))\coc_{\lie{P}}(h)\quad\forall g,h\in\pi_1(\Sigma).
\end{equation}

\begin{remark}\label{rk:tangent-rep}
The space of maps $\coc:\pi_1(\Sigma)\to\sl$ which satisfy \eqref{eq:coc}, is denoted here by $\Zgroup$. Recall \cite{go84} that $\Zgroup$ is identified to the tangent space of the representation variety at $\rho$.
Namely given a smooth path of representations $\rho_t$, the corresponding tangent vector at $\rho=\rho_0$ is encoded by the cocycle $\coc(g):=\frac{d\rho_t(g)\rho(g)^{-1}}{dt}\Bigr|_{t=0}$.
\end{remark}

    The tangent space of the representation variety contains the subspace of vectors tangent to the orbit of the $\PSL(2,\C)$-action given by conjugation. Such a subspace correspond to the space of coboundaries in $\Zgroup$ and is described in this way:
\begin{align*}
\Bgroup=&\{\lie{T}\in \Zgroup| \textrm{ there exists }\lie{T}_0\in\sl\\ &\textrm{ such that } \lie{T}(g)=\lie{T}_0-\mathrm{Ad}(\rho(g))\lie{T}_0\textrm{ for all }g\in\pi_1(\Sigma)\}\,.
\end{align*}
The tangent space of the character variety is then identified to the quotient 
\[
\Hgroup=\Zgroup/\Bgroup\,.
\]
We refer to \cite{go84} for further details.

\begin{defn}
A map $\lie{P}:\widetilde{\sTk{0}}\to\sl$ is said $\rho$-quasi-periodic if there exists a cocycle $\coc_{\lie{P}}\in \Zgroup$ such that
\eqref{eq:primitive equivariance}
holds. The cocycle $\coc_{\lie{P}}$ is said to be the \emph{period} of $\lie{P}$.
\end{defn}

Notice that elements in $\sC^0_\rho$ are exactly the quasi-periodic cochains with period equal to $0$.

We have shown that any $\rho$-equivariant $d$-closed element in $\sC^1_\rho$ is the differential of a quasi-periodic $0$-cochain. Conversely it is clear that the differential of a $\rho$-quasi-periodic $0$-cochain is a $\rho$-equivariant $d$-closed cochain.

So eventually the differential $d$ restricts to a surjective map from the space of $\rho$-quasi-periodic $0$-cochains of $\tilde\sC^0$ to $\rho$-equivariant  $d$-closed cochains.

This map is not injective, since $\lie{Q}$ admits many primitives. Indeed $\sl$ acts on $\tilde{\sC}^0$ by translations, and this action preserves the space of $\rho$-quasi-periodic cochains. In fact, the space of primitives of $\lie{Q}$ is an orbit for such an action.
Notice however that the period of $\lie{P}+\coc_0$ is equal to $g\mapsto\coc_{\lie{P}}(g)+\coc_0-\mathrm{Ad}\rho(g)\cdot \coc_0$. So periods of elements in the same $\sl$-orbits differ by coboundaries.

In this way a natural \emph{$d$ - period map} is defined:
\begin{align}\label{eq:period map}
   \per_d: H^1(\sC^*_\rho, d) &\to \Hgroup\\
[\lie{Q}] &\mapsto [\coc]\nonumber
\end{align}
where $\coc$ is the period of any $d$-primitive of $\lie{Q}$.

The following is well-known (see for instance \cite{Labourie-survey}).
\begin{prop}
    The period map \eqref{eq:period map} is an isomorphism.
\end{prop}

In a similar fashion we see that any $\delta$-closed $1$-cochain is the image through $\delta$ of a quasi-periodic $2$-cochain,  defined in a similar way.
Also in this case we have a surjective map from the space of quasi-periodic $2$-cochains to the space of $\rho$-invariant $\delta$-closed $1$-cochains. The fibers are the orbits of the $\sl$-action on $\tilde\sC^2$ by translation.
So also in this case a \emph{$\delta$-period map} sending $[\lie{Q}]\in H^1(\sC^*_\rho, \delta)$ to the period of a $\delta$-primitive of $\lie{Q}$, yields an isomorphism
\begin{align}\label{eq:delta-period map}
   \per_{\delta}: H^1(\sC^*_\rho, \delta) &\to \Hgroup
\end{align}

\section{Circle packings on a surface} \label{sec: circle packing preliminaries}
In this section we will introduce the notion of disk and circle packing on an arbitrary projective surface, that is a topological surface equipped with an $(\CP^1,\PSL(2,\C))$-structure. A circle packing is a collection of disks with some prescribed tangency conditions that are encoded by a triangulation called the {\it nerve} of the circle packing. We begin with a formal definition.

\begin{defn}
Let $\Pj$ be a projective surface with topological support  $\Sigma$.
An (open) \emph{disk} in $\Pj$ is a precompact open subset that is projectively 
equivalent to an open round disk in $\CP^1$. A closed disk in $\Pj$ is the 
closure of a disk.
\end{defn}
When it is not specified, in this paper disks will be considered open. Open disks are always embedded by definition.

Notice that closed disks are always considered 
compact in this paper, but not necessarily embedded. However, we will see that they are always embedded at least when the projective surface is simply connected.

\begin{lemma}\label{lem:disk in simply connected}
If $\Pj$ is \emph{simply connected}, then any closed disk  in $\Pj$ is projectively equivalent to a closed round disk in $\CP^1$. 
Moreover, if $D_1,D_2$ are two disks with disjoint interiors, then either their closures are disjoint or their closures meet at a single point or they are complementary hemispheres in $\CP^1$.
\end{lemma}
\begin{proof}
Let $D$ be an (open) disk.
By the essential uniqueness of the developing map, the restriction of $\dev$ to the interior of $D$ is a homeomorphism onto its image, and that image is an open disk $\Delta$ of $\CP^1$. Since $\overline{D}$ is compact,  we have $\dev(\overline{D})=\overline{\Delta}$.

We aim to prove that the restriction of $\dev$ to $\overline{D}$ is injective.

First notice that $\dev(\overline{D}\setminus D)=\partial\Delta$: indeed if $p\in\overline{D}\setminus D$, 
then there is $p_n\in D$ converging to $p$.
The sequence $p_n$ is diverging in $D$, so $\dev(p_n)$ is a diverging sequence in $\Delta$ and we conclude that 
$\dev(p)=\lim\dev(p_n)$ lies in $\partial\Delta$. Conversely every $\mathsf{x}\in\partial\Delta$ is a limit of $\dev(p_n)$ for some sequence diverging $p_n\in D$. Up to passing to a subsequence, $p_n$ converges to a point $p$ in $\overline{D}\setminus D$ and $\dev(p)=\mathsf{x}$.

Let us take distinct $p,q\in \overline{D}$ and let us prove that $\dev(p)\neq\dev(q)$. This is clear if both $p,q\in D$, and also if $p\in D $ and $q\in\overline{D}\setminus D$. In fact, in the latter case $\dev(p)\in\Delta$ while $\dev(q)\in\partial\Delta$.

So it remains to consider the case where both $p,q\in\overline{D}\setminus D$.
Assume by contradiction that they have the same image, say
$\mathsf{x}=\dev(p)=\dev(q)$. Take two disjoint open neighborhoods $U_1$ and $U_2$ in $\Pj$ of $p$ and $q$ respectively such that $\dev|_{U_i}$ is a homeomorphism onto an open subset $V_i$ of $\CP^1$, with $V_i\cap\Delta$ connected.
We claim that $\dev(U_i\cap D)=V_i\cap\Delta$.
Clearly $\dev(U_i\cap D)$ is an open subset of $V_i\cap\Delta$. 
So it is sufficient to prove that any accumulation point $\mathsf{y}_\infty$ of
$\dev(U_i\cap D)$ which is contained in $V_i\cap\Delta$ is in fact a point of
$\dev(U_i\cap D)$.
Notice that we  have $\mathsf{y}_\infty=\lim\dev(p_n)$ with $p_n\in U_i\cap D$.
Since $\dev(p_n)$ is converging in $V_i$, then $p_n$ is converging to a 
point $p'\in U_i\cap\overline{D}$. 
\begin{figure}[htbp]
  \centering
  \includegraphics[width=10cm]{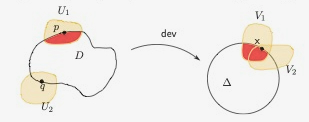}
  \caption{If $\dev(p)=\dev(q)=\mathsf{x}$, then $V_1\cap V_2\cap\Delta\neq\varnothing$. Since $\dev(U_1\cap D)=V_1\cap D$ (the red regions in the picture), and $\dev(U_2\cap D)=V_2\cap\Delta$, we conclude that $\dev$ restricted to $D$ is not injective.}
\end{figure}
If $p'\in U_i\cap\partial D$, then we would have $\mathsf{y}_\infty=\dev(p')\in\partial\Delta$, which contradicts the assumption that $\mathsf{y}_\infty$ lies in $V_i\cap\Delta$. This shows that $p'\in U_i\cap D$, so $\mathsf{y}_\infty\in\dev(U_i\cap D)$.

So $\dev(U_1\cap D)=V_1\cap\Delta$ and $\dev(U_2\cap D)=V_2\cap\Delta$, but $V_1$ and $V_2$ are open subsets of $\CP^1$ meeting at $\mathsf{x}\in\partial\Delta$.
Thus $V_1\cap V_2\cap\Delta\neq\varnothing$.
So if $\mathsf{y}$ lies in this intersection it is the image of $U_1\cap D$
and $U_2\cap D$. But $U_1\cap U_2=\varnothing$ and so therefore $U_1\cap D$ and $U_2\cap D$ are disjoint subsets of $D$: we are then left contradicting the injectivity of $\dev$ on $D$.

Clearly $\dev(U\cap D)$ is an open subset of $V\cap\Delta$. 

Let us prove now that the closure of two disjoint open disks in $\Pj$ can meet only at one point. If $p\in \overline{D_1}\cap \overline{D_2}$, and $U$ is a neighborhood of $p$ in $\Pj$ on which $\dev$ is injective, then we have that $V_i=\dev(\overline{D_i}\cap U)$ are  neighborhoods of $\dev(p)$ in the closed disk $\overline{\Delta_i}=\dev({\overline D_i})$. Since $\dev(D_i\cap U)=\Delta_i\cap V_i$ are disjoint, then 
$\overline{\Delta_1}$ and $\overline{\Delta_2}$ meet tangentially at $\dev(p)$. But there is a dichotomy for pairs of disjoint disks meeting tangentially in $\CP^1$: if $\Delta_1$ is the upper-half plane, and the pair meet at the origin, then the other disk $\Delta_2$ is either the lower half-plane or a disk properly embedded in the lower half-plane. In the former case, the developed images $\Delta_1$ and $\Delta_2$ would share a common boundary, and each boundary point would be a limit point of developed images of sequences in each of $D_1$ and $D_2$.  Thus the union of $\overline{D_1}$ and $\overline{D_1}$ would admit a manifold structure with an atlas of charts covering each $D_i$ together with boundary charts for $\overline{D_1}$ and $\overline{D_1}$ whose union provides a chart on each boundary point.  Altogether this defines a manifold comprising the union of $D_1$ and $D_2$, glued along the common boundary, hence the sphere in the statement of the theorem.

In the latter case, since
$\dev(\overline{D_1}\cap\overline{D_2})\subset \overline{\Delta_1}\cap\overline{\Delta_2}$ and $\dev$ is injective over $\overline{D_1}$ with $\overline{\Delta_1}\cap\overline{\Delta_2}$ being a singleton, we conclude that 
$\overline{D_1}$ and $\overline{D_2}$ meet only at $p$.
\end{proof}

\begin{figure}[htbp]
  \centering
  \includegraphics[width=5cm]{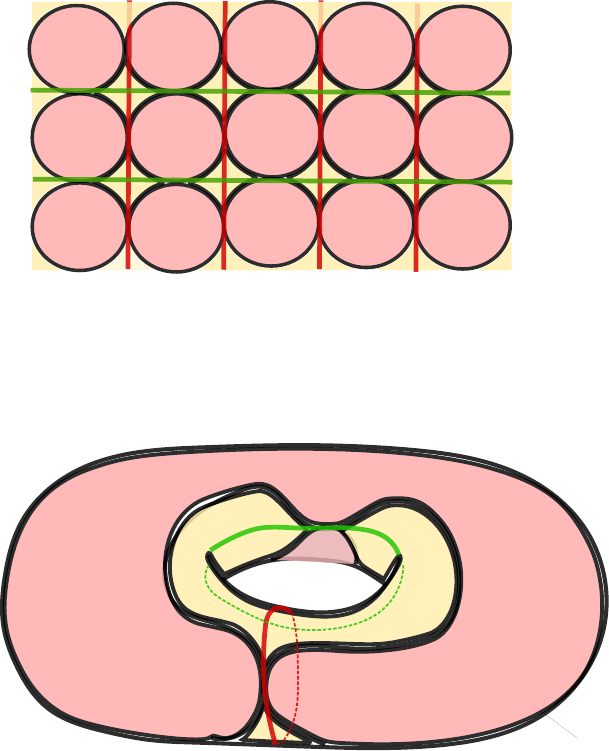}
  \caption{An example of a non-embedded disk in a translation torus, and its universal covering.}
\end{figure}
\begin{cor}\label{cor:self-intersection boundary disk}
Let $D$ be an open disk in a projective surface $\Pj$.
Then there exists a projective morphism $f:\overline\Delta\to\overline D$ i.e. a map which is locally projective when expressed in projective coordinates, with the properties that
\begin{itemize}
\item $f$ is injective on $\Delta$;
\item $\partial D=f(\partial \Delta)$ is an immersed curve, which passes through a given point no more than twice, always with the same tangent line.
Moreover the set of points through which it passes more than once is discrete.
\end{itemize}
\end{cor}
\begin{proof}
Let $\pi:\tilde{\Pj}\to\Pj$ be the universal covering and let 
$\tilde D$ be a component of $\pi^{-1}(D)$. We have that $\tilde D$ is 
an open disk, so there exists a projective isomorphism 
$\tilde f:\overline{\Delta}\to\overline{\tilde D}$.  Let us put $f=\pi\circ \tilde f$.
Since $D$ is simply connected, the restriction $\pi|_{\tilde D}:\tilde D\to D$ is a homeomorphism, so the restriction of $f$ to $\Delta$  realizes a homeomorphism between  $\Delta$ and $D$.

The orbit of $\tilde D$ under the action of $\pi_1(\Sigma)$
is a disjoint union of open disks. However there may exist a finite number of $g_i\in\pi_1(\Sigma)$ such that
$g_i(\partial \tilde D)\cap(\partial \tilde D)\neq\emptyset$.
By the second part of Lemma \ref{lem:disk in simply connected}, we have that $g_i(\partial \tilde D)\cap\partial\tilde D$ consists of one point $\tilde p_i$, and the boundaries meet tangentially. Since those points correspond to auto-intersection points of $\partial D$ we conclude that the tangency points are discrete and each tangency point has a common tangency line. Finally notice that three disks cannot meet tangentially at a single point in $\tilde{\Pj}$, so the curve $\partial D$ cannot pass through the same point more than twice.
\end{proof}

\begin{defn}
A \emph{triangular region} in $\CP^1$ is an open region bounded by three arcs of circles that pairwise meet each other tangentially. 
The circular arcs are called {\it edges} of the triangular region, while the intersection points between two edges are called vertices. 
\end{defn}

Notice that the orientation on $\CP^1$ induces a natural cyclic order on the set of vertices of a triangular region.

We will use in the paper the following elementary fact and its consequences.

\begin{lemma}
Let $\mathsf{T}$ be a triangular region in $\CP^1$ with vertices $\mathsf{p}_1,\mathsf{p}_2,\mathsf{p}_3$. Denote by $\mathsf{C}$ the circle passing through $\mathsf{p}_1,\mathsf{p}_2,\mathsf{p}_3$, and $\Delta$ be the disk bounded by $\mathsf{C}$ and containing $\mathsf{T}$.

Each edge of $\mathsf{T}$  meets the circle $\mathsf{C}$ orthogonally. So $\mathsf{T}$ is an ideal hyperbolic triangle with respect to the Poincar\'e metric on $\Delta$.
\end{lemma}

\begin{proof}
    We can take $\Delta$ to the the upper half plane and one of the points to be at infinity.  The tangency condition of the edges of the triangle that are incident to the point at infinity forces those edges to be parallel straight lines.  But then the circle extending the circular edge which connects the finite points must have parallel tangent lines at those finite points:  this can only happen for antipodal points on a circle and so the real axis is a diameter of that circle, and each line connecting a finite point to infinity must be vertical.  This is the content of the lemma for this choice of $\Delta$.
\end{proof}

\begin{cor}
 A triangular region in $\CP^1$ is uniquely determined by its vertices equipped with the induced cyclic order.
\end{cor}

The reader may wish to note that three circles embedded in $\CP^1$ determine, in their complement, two triangular regions with identical vertices, but different orderings.

A \emph{lune} is an open region of $\CP^1$ bounded by two circular arcs meeting at two points, called vertices of the lune, and forming an interior angle of $\pi/2$ on both vertices.
Notice that the disk $\Delta$ containing a triangular region $\mathsf{T}$ decomposes as the union of the closure of $\mathsf{T}$ and three lunes.

A triangular region of a projective surface $\Pj$ is a region $T$ that is projectively equivalent to a triangular region of $\CP^1$. Similarly a lune on $\Pj$ is a region that is projectively isomorphic to a lune in $\CP^1$.

\begin{defn}
    Let $\Pj$ be a projective surface. A {\it circle packing} on $\Pj$ is a locally finite collection $\cpack=\{D_\alpha\}$ of mutually disjoint open disks, such that each connected component of $\Pj\setminus\overline{\bigcup_{\alpha}D_\alpha}$ is a (open) triangular region.
\end{defn}

\begin{remark}
 We emphasize that in this paper, in contrast to some other papers, we only allow one projective type of complementary region. This will have consequences for the topology of the dual graph.    
\end{remark}

Notice that the lifting $\widetilde{\cpack}$ of a 
circle packing $\cpack$ to the universal covering $\tilde\Pj$ is a circle packing, whose elements are permuted by the action of the fundamental group of the underlying surface $\Sigma$. Moreover the induced action of $\pi_1(\Sigma)$ on $\widetilde\cpack$ is free.

In this way we see there exists a bijective correspondence between circle packings on $\Pj$ and circle packings on  $\tilde\Pj$ on which $\pi_1(\Sigma)$ freely acts.

A natural graph can be associated to every circle packing $\cpack$ on $\Pj$. Namely one chooses a point $q_\alpha$ in each disk $D_\alpha$.
If $\overline{D_\alpha}$ and $\overline{D_\beta}$ meet at (distinct) points $p_1,\ldots,p_k$, then connect $q_\alpha$ to $q_\beta$ by disjoint arcs $\ed_1,\ldots\ed_k$ contained in $\overline{D_\alpha\cup D_\beta}$ and passing through the points $p_1,\ldots, p_k$ respectively. Clearly if $\alpha=\beta$, points $p_i$ correspond to possibly self-intersections of the boundary of the disk, as described in Corollary \ref{cor:self-intersection boundary disk}.
\begin{figure}[htbp]
  \centering
  \includegraphics[width=8cm]{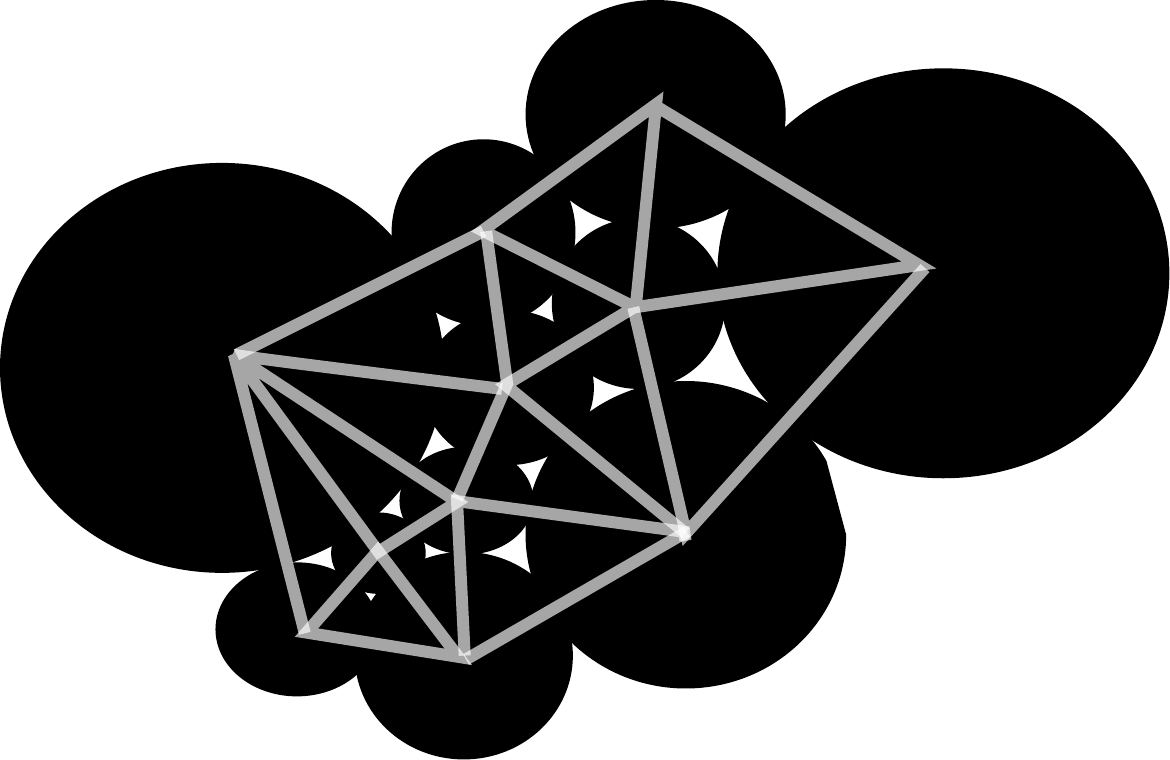}
  \caption{The nerve of a circle packing.}
\end{figure}

Recall the notion of quasi simplicial from Definition~\ref{defn:quasi simplical}.

\begin{lemma}
The graph obtained in this way is the $1$-skeleton of a quasi-simplicial triangulation $\sT=\sT(\cpack)$ of $\Pj$.
\end{lemma}
\begin{proof}
The fact that $\sT$ is a triangulation follows by construction. Let us prove that it is quasi-simplicial.

Notice that the universal covering $\widetilde{\sT}$ of $\sT$ corresponds to the graph associated to the circle packing $\widetilde{\cpack}$.
Since by Lemma \ref{lem:disk in simply connected} a disk in $\tilde\Pj$ has no self-intersections, we see $\widetilde{\sT}$ contains no closed edge. Moreover since two tangent disks in $\tilde\Pj$ meet at most at one point (again by Lemma \ref{lem:disk in simply connected}), we deduce that $\widetilde{\sT}$ contains no parallel edges. So $\widetilde{\sT}$ is simplicial, and by Lemma \ref{lm:quasi-simplicial}, $\sT$ is quasi-simplicial.
\end{proof}

The triangulation $\sT(\cpack)$ is said to be the {\it nerve} of the circle 
packing. It codifies the information about the combinatorics of the 
circle packing.
Notice that 
\begin{itemize}
    \item vertices of $\sT(\cpack)$ bijectively correspond to disks of $\cpack$;
    \item  edges of $\sT(\cpack)$ bijectively correspond to intersection points between the closure of the disks of the circle packing (including self-intersections);
    \item  faces of $\sT(\cpack)$ are in bijective correspondence with the triangular regions.    
\end{itemize}

Given a circle packing $\cpack$ on a projective surface $\Pj$,  the triangulation associated to $\tilde\cpack$ is clearly the lifting of $\sT(\cpack)$, say $\tilde\sT$. Notice that there is a natural action of $\pi_1(\Sigma)$ on $\tilde\sT$.

Fixing a developing map $\dev:\tilde\Pj\to\CP^1$, we define a map
\[
\sel_\cpack:\widetilde{\sTk{0}}\to\dS
\]
so that $\sel_\cpack(\tilde\alpha)$ is the point in de Sitter space corresponding to the disk $\dev(D_{\tilde\alpha})$.

As an immediate application of Proposition \ref{pr:ds and disks} and Lemma \ref{lm:ds tangent} we have the following.

\begin{lemma}
Denote by $\hol:\pi_1(\Sigma)\to\PSL(2,\C)$ the holonomy of the developing map $\dev$.
The map $\sel_\cpack$ satisfies the following properties:
\begin{itemize}
\item For all $g\in\pi_1(\Sigma)$ and for all $\alpha\in\widetilde{\sTk{0}}$ we have
\[
   \sel_\cpack(g\cdot\alpha)=\hol(g)\cdot\sel_{\cpack}(\alpha)\,.
\]
\item For all $\oed\in\widetilde{\sTor}$ we have
$\mink{\sel_{\cpack}(\oed_-), \sel_{\cpack}(\oed_+)}=-1$, and the lightlike vector $\sel_{\cpack}(\oed_-)+ 
\sel_{\cpack}(\oed_+)$ is future directed.
\end{itemize}
\end{lemma}

The map $\sel_{\cpack}$ is called the {\it selection} map associated to the circle packing. Notice that the map $\sel_{\cpack}$ actually depends on the choice of a developing map for $\Pj$. However since two developing maps differ by post-composition of some element in $\PSL(2,\C)$, it is easy to check that the map $\sel_{\cpack}$ is well defined up to the action of $\PSL(2,\C)$ on $\dS$.

\section{Circle packing with a fixed nerve}\label{sec:fixed nerve}
In this section we fix an oriented closed topological surface $\Sigma$ of genus $\genus(\Sigma)\geq 2$.
A marked projective surface is a projective surface $\Pj$ equipped with a marking $f_{\Pj}:\Sigma\to\Pj$, that is an orientation preserving diffeomorphism.
A marked projective structure with circle packing is a pair $(\Pj,\cpack)$ where $\Pj$ is a marked projective surface, and $\cpack$ is a circle packing on $\Pj$. 
Two marked projective surfaces  with circle packing $(\Pj,\cpack)$, $(\Pj',\cpack')$ are isomorphic (through a homotopically trivial map) 
if there exists a projective isomorphism $F:\Pj\to\Pj'$, such that
\begin{enumerate}
    \item $F$ induces a bijective correspondence between disks of $\cpack$ and disks of $\cpack'$.
    \item $F\circ f_{\Pj}$ is homotopically equivalent to $f_{\Pj'}$.
\end{enumerate}
Given a marked projective structure with circle packing $(\Pj,\cpack)$,
we can use the marking to pull back the nerve of the circle packing to a triangulation on $\Sigma$. Such a triangulation is well-defined up to isotopy and only depends on the isomorphism class of the circle packing. In the following we will always implicitly identify the nerve of the circle packing with a triangulation of $\Sigma$ in this way.
The aim of this section is to construct in some explicit way the moduli space $\sP_{\sT}$ of marked projective structures with circle packing whose nerve is a fixed quasi simplicial triangulation $\sT$. 

This space is not empty, since by the Circle Packing Theorem, it contains a unique element $(\Pj_{0},\cpack_{0})$, with $\Pj_{0}$ Fuchsian.

The main tool in our investigation will be the selection map $\sel_{\cpack}$ introduced in the previous section. 

Let us introduce some notation. 
We denote by $\hat\sM_{\sT}$ the space of pairs $(\rho,\sel)$ where
$\rho:\pi_1(\Sigma)\to\PSL(2,\C)$ is a non-elementary representation, and $\sel:\widetilde{\sTk{0}}\to\dS$ is a $\rho$-equivariant map.

Denote by $\mathcal R$ the (smooth) variety of non-elementary representations of $\pi_1(\Sigma)$ into $\PSL(2,\C)$.
Notice that once a family $\sF$ of lifting of vertices of $\sT$ into $\widetilde{\sT}$ is chosen, the map
$\hat\sM_{\sT}\to\sR\times(\dS)^{\sF}$ defined by
\[
   (\rho,\sel)\mapsto(\rho,\sel|_{\sF})
\]
is a bijection that induces on $\hat\sM_{\sT}$ a structure of differentiable manifold.
Indeed it is immediate to see that such a structure does not depend on the choice of the lifting of the vertices.
We conclude:
\begin{lemma} \label{lem:M-hat_T}
 The space
    $\hat\sM_{\sT}$ is a manifold of dimension $12\genus(\Sigma)-6+3v$ so that the natural forgetful map $\hat\sM_{\sT}\to\mathcal R$ is a fiber bundle with fiber equal to $(\dS)^{v}$.
\end{lemma}

The group $\PSL(2,\C)$ acts on $\hat\sM_{\sT}$ by $A\cdot(\rho,\sel)=(A\rho A^{-1}, A\cdot\sel)$. Since the action of $\PSL(2,\C)$ on $\mathcal R$ is proper (see Chapter 5 of \cite{Labourie-survey}), it follows that so is the action on $\hat\sM_{\sT}$. In particular on $\sM_{\sT}:=\hat\sM_{\sT}/\PSL(2,\C)$ we may define a structure of a manifold that makes the projection $\mathbf{p_\sM}:\hat\sM_{\sT}\to\sM_{\sT}$ a submersion. If $\chi$ denotes the character variety of $\pi_1(\Sigma)$, the  map $[\rho,\sel]\mapsto [\rho]$ of $\sM_{\sT} \to \chi$ is smooth.

Let us notice that for a  projective surface with a circle packing $(\Pj,\cpack)$  with nerve $\sT$, the element $[\hol_{\Pj},\sel_{\cpack}]$ is well defined in the quotient $\sM_\sT$ (while the representative depends on the choice of a developing map). The aim of this section is to prove that such a correspondence allows us to identify, with a submanifold in $\sM_\sT$, the moduli space of projective structures with circle packing with nerve $\sT$.

We consider next a map
\[
\hat{\tgmap}:\hat\sM_{\sT}\to\R^{\sTk{1}}
\]
defined as follows.
Given an edge $\ed$ we consider any lift $\tilde\ed$ in $\tilde\sT$
and denote by $\tilde\alpha$, $\tilde\beta$ its endpoints.
Then let us put
\begin{equation} \label{defn:I-hat}
    \hat\tgmap(\rho,\sel)(\ed)=\mink{\sel(\tilde{\alpha}),\sel(\tilde{\beta})}+1\,.
\end{equation}

Using that if $\tilde\ed'$ is a different lifting of $\ed$ there exists $g\in\pi_1(\Sigma)$ such that $\tilde\ed'=\hol(g)\tilde\ed$, it is easy to check that $\hat\tgmap$ is well defined and does not depend on the choice of the lifts of the edges.
Moreover $\hat\tgmap$ is invariant under the action of $\PSL(2,\C)$ on $\hat\sM_{\sT}$, so it induces a smooth map $\tgmap:\sM_{\sT}\to\R^{\sTk{1}}$.

We remark that if $(\Pj,\cpack)$ is a 
projective structure with a circle 
packing whose nerve is $\sT$, then by general facts $\hol_{\Pj}$ is non-elementary  (see \cite{Dumas:SurveyCxProjectiveStructures}), and the corresponding element
$[\hol_{\Pj},\sel_{\cpack}]$
lies in the zero set locus of $\tgmap$.
However a priori if $\hat\tgmap(\rho,\sel)=0$, it is possible that for $\alpha,\beta$ vertices of an edge $\ed\in\widetilde{\sTk{1}}$, the vector $\sel(\alpha)+\sel(\beta)$ (which is inevitably lightlike) is past directed and the corresponding disks are not tangent. To avoid these cases, let us consider the  subset  $\hat\sM_*$ of  $\hat\sM_{\sT}$ which contains the pairs $(\rho,\sel)$ such that for any edge $\ed\in\widetilde{\sTk{1}}$ with endpoints $\alpha,\beta$ we have that $\sel(\alpha)+\sel(\beta)$ is not a (semi-)negative Hermitian matrix.
The equivariance of $\sel$ ensures that this last condition can be checked on a finite family of edges, so $\hat\sM_*$ is an open subset on $\hat\sM_{\sT}$.
Moreover $\hat\sM_*$ is invariant under the $\PSL(2,\C)$-action so it projects to an open subset $\sM_*$ of $\sM_{\sT}$.

Denote by $\hat\sP_{\sT}$ the set $\hat\tgmap^{-1}(0)\cap\hat{\sM}_*$.
This set is invariant under the action of $\PSL(2,\C)$, and its projection is the subset $\sP_\sT:=\tgmap^{-1}(0)\cap\sM_*$.
We remark that by Lemma \ref{lm:tangent point}, a pair $(\rho,\sel)$ lies in $\hat\sP_\sT$ if and only if the disks $\Delta(\sel(\alpha))$ and $\Delta(\sel(\beta))$ are tangent for all vertices $\alpha,\beta$ of some edge of $\widetilde{\sT}$.

Notice that for any pair $(\Pj,\cpack)$, the element $[\hol,\sel_\cpack]\in\sM_{\sT}$ lies in $\sP_{\sT}$.
We prove now that every element in $\sP_{\sT}$ arises in this way.

\begin{prop}\label{pr: from selection map to circle packing}
Given any $[\rho,\sel]\in\sP_{\sT}$ there exists a projective structure $\Pj$ equipped with a circle packing $\cpack$ such that the holonomy of $\Pj$ is $\rho$ and $\sel_{\cpack}=\sel$.

\end{prop}
\begin{proof}
    Let $\tau$ be a face of  $\tilde\sT$ with vertices $\alpha_1,\alpha_2,\alpha_3$, where the numbering is taken so that
    $\alpha_{i+1}$ is seen on the left of $\alpha_i$ from a point of $\tau$.
    
    The three disks $\Delta(\sel(\alpha_1)), \Delta(\sel(\alpha_2)), \Delta(\sel(\alpha_3))$ are mutually tangent, so their complement is the union of two triangular regions. To see this, realize one of the disks as the upper halfplane so that the other two disks, being pairwise tangent, are disks tangent to each other in the lower halfplane tangent to the real axis.
    We consider the triangular region $\mathsf{T}(\tau)$ such that $\Delta_{i+1}$ is seen on the left of $\Delta_{i}$ from a point of $\mathsf{T}(\tau)$.
    Denote by $\sel^*(\tau)$ the point of $\dS$ dual to the disk containing $\mathsf{T}(\tau)$ as an ideal triangle.

    By construction we see that for every face $\tau$ and every vertex $\alpha$ of $\tau$ the intersection 
    $\mathsf{L}(\alpha,\tau)=\Delta(\sel(\alpha))\cap\Delta(\sel^*(\tau))$ is a projective lune.
    
    Notice that we have that $\mathsf{L}(\alpha_1,\tau), \mathsf{L}(\alpha_2,\tau), \mathsf{L}(\alpha_3,\tau)$ are mutually disjoint lunes in $\Delta(\sel^*(\tau))$. Similarly if $\tau_1,\ldots,\tau_k$ are the faces adjacent to a vertex $\alpha$, then 
    $\mathsf{L}(\alpha,\tau_i)$ are mutually disjoint lunes in $\Delta(\sel(\alpha))$.

    Consider now 
    \[
     \mathfrak{P}=\bigsqcup_{\alpha\in\widetilde{\sTk{0}}} \Delta(\sel(\alpha))\,\sqcup\,\bigsqcup_{\tau\in\widetilde{\sTk{2}}}\Delta(\sel^*(\tau))\,,
    \]
    and denote by $\iota:\mathfrak{P}\to\CP^1$ the map that restricts to the natural inclusion on each component. Notice that $\mathfrak{P}$ is naturally a projective surface and $\iota$ can be regarded as its developing map.

    On $\mathfrak{P}$ we can consider the following relation:
    $\mathsf{x}\in\Delta(\sel(\alpha))\sim \mathsf{y}\in\Delta(\sel^*(\tau))$ if $\alpha$ is a vertex of $\tau$ and $\iota(\mathsf{x})=\iota(\mathsf{y})$.
    This relation basically identifies the copy of $\mathsf{L}(\alpha,\tau)$ contained in  $\Delta(\sel(\alpha))$ with the corresponding copy contained in $\Delta(\sel^*(\tau))$.

    Since lunes contained in a fixed disk are pairwise disjoint, it is easy to check that $\sim$ defines an equivalence relation. Denote by $\tilde{\Pj}^*$ the quotient $\mathfrak{P}/\sim$ and denote by $\pi_{\mathfrak{P}}:\mathfrak{P}\to\tilde{\Pj}^*$ the canonical projection.

    It is simple to check that
    \begin{itemize}
        \item the restriction of $\pi_{\mathfrak{P}}$ on each component of $\mathfrak{P}$ is a homeomorphism on an open subset of $\tilde{\Pj}^*$. Let us put 
        $D_\alpha:=\pi_{\mathfrak{P}}(\Delta(\sel(\alpha))$ and $D^*_\tau=\pi_{\mathfrak{P}}(\Delta(\sel^*(\tau))$.
        \item The relation $\sim$ is closed, so the quotient $\tilde{\Pj}^*$ is Hausdorff.
        \item The map $\iota$ factors through the projection $\pi_{\mathfrak{P}}$ to a map $\mathsf{j}:\tilde{\Pj}^*\to\CP^1$ that is a local homeomorphism.
        \item The map $\mathsf{j}$ restricts to a homeomorphism between  $D_\alpha$ (resp. $D^*_\tau$) and $\Delta(\mathsf{S}(\alpha))$ (resp. $\Delta(\mathsf{S}^*(\tau))$.
    \end{itemize}

Those properties imply that $\tilde{\Pj}^*$ is a surface and that the map $\mathsf{j}$ can be regarded as a developing map of a projective structure on $\tilde{\Pj}^*$.
\begin{figure}[htbp]
  \centering
  \includegraphics[width=5cm]{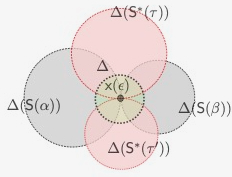}
  \caption{The disk $\Delta$ is chosen so that it meets only the lunes around $\mathsf{x}(\epsilon)$.}\label{fig:tangency}
\end{figure}
Let us remark that $\tilde{\Pj}^*$ contains many properly embedded punctured disks. Indeed let $\ed$ be an edge of $\widetilde{\sT}$ with endpoints $\alpha,\beta$, and denote by $\tau,\tau'$ its adjacent faces. Then the boundaries of the disks 
$\Delta(\sel(\alpha)), \Delta(\sel(\beta)), \Delta(\sel^*(\tau)), \Delta(\sel^*(\tau'))$
all meet at a point $\mathsf{x}(\ed)$. 
Notice that $\mathsf{x}(\ed)$ does not belong to 
$ \mathcal V=\mathsf{j}(D_\alpha\cup D_{\beta}\cup D_\tau\cup D_{\tau'})=\Delta(\sel(\alpha))\cup\Delta(\sel(\beta))\cup\Delta(\sel^*(\tau))\cup\Delta(\sel^*(\tau'))$, but  $\mathcal V$ covers a small punctured disk $\Delta\setminus\{\mathsf{x}(\ed)\}$ centered at $\mathsf{x}(\ed)$. 
We can choose the disk $\Delta$ so that the only lunes of $\mathcal V$ that meet $\Delta$ are $\mathsf{L}(\alpha,\tau), \mathsf{L}(\alpha,\tau'), \mathsf{L}(\beta,\tau), \mathsf{L}(\beta,\tau')$, see Figure \ref{fig:tangency}. 

Denote by $D^*(\ed)$ the component of $\mathsf{j}^{-1}(\Delta)$ intersecting $D_\alpha$; then $\mathsf{j}$ restricted to $D^*(\ed)$ is injective and its image is $\Delta\setminus\{\mathsf{x}(\ed)\}$,
so $D^*(\ed)$ is a punctured disk, and we can then fill in the puncture, gluing $\Delta$ to $\tilde{\Pj}$ by identifying $D^*(\ed)$ to $\Delta\setminus\{\mathsf{x}(\ed)\}$.
Let us denote by $\tilde{\Pj}$ the surface obtained by filling the puncture in 
$D^*(\ed)$ for each $\ed\in\widetilde{\sTk{1}}$, and we denote by $p(\ed)$ the filling point.

Notice that the map $\mathsf{j}$ extends to a local homeomorphism $\mathsf{j}$ which identifies $D(\ed)$ to $\Delta$.

By construction $D_\alpha$ is an open disk in $\tilde{\Pj}$.
We have that $\overline{D_\alpha}\cap\overline{D_\beta}=\{p(\ed)\}$ if $\alpha$ and $\beta$ are the endpoints of an edge $\ed\in\widetilde{\sTk{1}}$, otherwise $\overline{D_\alpha}\cap\overline{D_\beta}=\emptyset$. Moreover the complement of the union $\cup_{\alpha}\overline{D_{\alpha}}$ of the closures of $D_\alpha$ is the union of the triangular regions obtained by removing in each $D^*_\tau$ the corresponding three lunes. 
We conclude that the family $\widetilde{\cpack}:=\{D_\alpha|\alpha\in\widetilde{\sT}\}$ defines a circle packing over $\widetilde{\Pj}$.

Notice that the family 
$\{D_\alpha|\alpha\in\widetilde{\sTk{0}}\}\cup\{D^*_\tau|\tau\in\widetilde{\sTk{2}}\}\cup\{D(\ed)|\ed\in\widetilde{\sTk{1}}\}$ forms a good covering of $\tilde{\Pj}$. A simple analysis shows that the nerve of such a covering is the barycentric subdivision of $\widetilde{\sT}$, so we see that $\widetilde{\Pj}$ is homotopically equivalent to $\widetilde{\Sigma}$.

The group $\pi_1(\Sigma)$ naturally acts on $\mathfrak{P}$. Indeed for any $g\in\pi_1(\Sigma)$ we can consider the homeomorphism
$\mathfrak{r}(g):\mathfrak{P}\to\mathfrak{P}$ that sends $\Delta(\sel(\alpha))$ (resp. $\Delta(\sel^*(\tau))$) to $\Delta(\sel(g\cdot\alpha))$ (resp. $\Delta(\sel^*(g\cdot\tau))$) defined so that $\iota\circ \mathfrak{r}(g)=\rho(g)\circ\iota$. The map $\pi_1(\Sigma)\times\mathfrak{P}\to\mathfrak{P}$ sending $(g,x)$ to $\mathfrak{r}(g)(x)$ defines a left action of $\pi_1(\Sigma)$ over $\mathfrak{P}$

Since the map $\mathfrak{r}(g)$ preserves the equivalence classes of $\sim$, that map $\mathfrak{r}(g)$ descends to a map $r(g):\tilde{\Pj}\to\tilde{\Pj}$, defining an action of $\pi_1(\Sigma)$ on $\tilde{\Pj}$.
By construction we have $\mathsf{j}\circ r(g)=\rho(g)\circ\mathsf{j}$.
This action is proper: indeed for any $\alpha\in\widetilde{\sTk{0}}$ (resp. $\tau\in\widetilde{\sTk{2}}$) we have that $g\cdot D_\alpha\cap D_\alpha=\emptyset$ (resp. $g\cdot D_\tau^*\cap D_\tau^*=\emptyset$) for $g\neq 1$.
Thus we can consider the surface $\Pj=\tilde{\Pj}/\pi_1(\Sigma)$.
Since $\tilde{\Pj}$ is simply connected, we see that $\Pj$ is a closed surface diffeomorphic to $\Sigma$.
That surface $\Pj$ inherits a projective structure from $\tilde{\Pj}$ for which the developing map coincides with $\mathsf{j}$ and the holonomy map with $\rho$.
Since the circle packing $\{D_\alpha|\alpha\in\widetilde{\sTk{0}}\}$ is invariant under the action of the group, that packing descends to a circle packing $\cpack$ on $\Pj$. 

Finally the construction allows us to define an isomorphism between $\pi_1(\Sigma)$ and $\pi_1(\Pj)$ that is well-defined up to inner automorphisms. This isomorphism allows us to construct a marking $\Sigma\to\Pj$ that is well-defined up to homotopically trivial diffeomorphisms, so that $(\Pj,\cpack)$ can be considered as a marked projective structure with a circle packing over $\Sigma$.

 By construction, the nerve of $\cpack$ is identified to $\sT$, and, under this canonical identification, the selection map $\sel_{\cpack}$ is $\sel$.
\end{proof}

 In the proof we have implicitly constructed a projective structure $\Pj$ equipped with a circle packing $\cpack$ {\bf and} a \emph{dual configuration}, that is a family of disjoint disks $D^*_\tau$, one for each face $\tau$ of $\widetilde{\sT}$, such that $D^*_\tau$  meets exactly the disks $D_\alpha$ corresponding to the vertices of $\tau$. 
\begin{figure}[htbp]
  \centering
  \includegraphics[width=13cm]{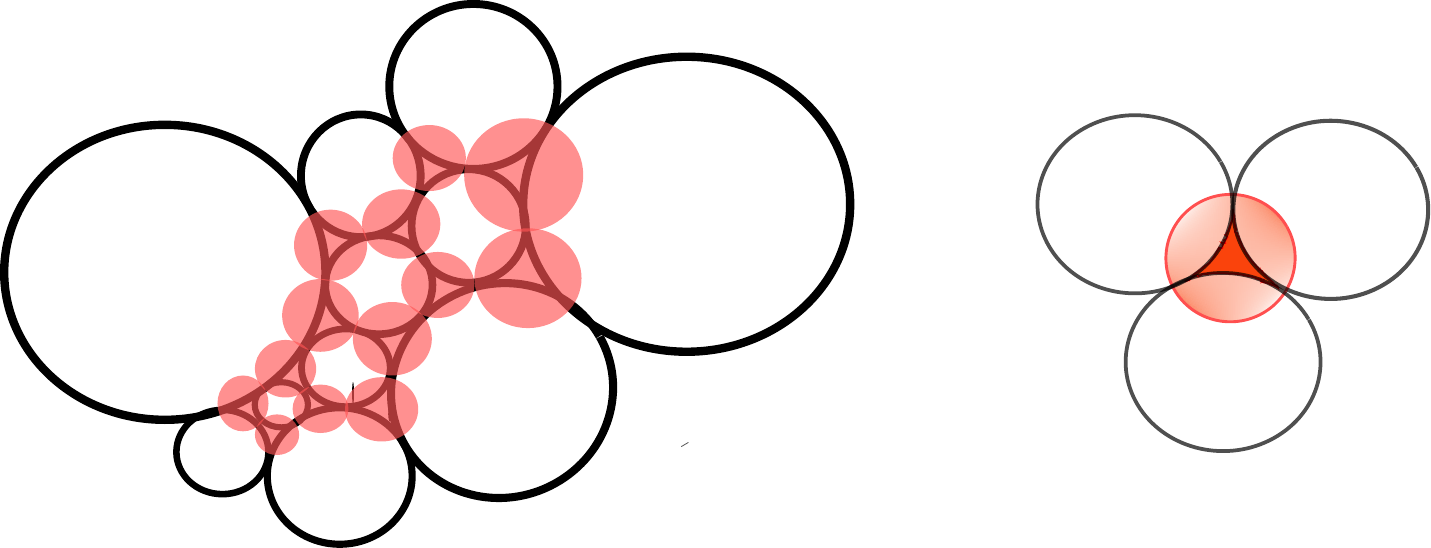}
  \caption{The dual configuration. Each disk $D^*_\tau$ is the union of a triangular region and three lunes.}
\end{figure}
Indeed, such a dual packing configuration can be constructed for any circle packing $(\Pj,\cpack)$. In particular if $\tau$ is a face of $\widetilde{\sT}$ with vertices $\alpha_1,\alpha_2,\alpha_3$, there is a  triangular 
region $T$ in $\Pj$ bounded by $\overline{D_{\alpha_1}}\cup\overline{D_{\alpha_2}}\cup\overline{D_{\alpha_3}}$. It is easy to check that the restriction of $\dev$ to $U=T\cup\overline{D_{\alpha_1}}\cup\overline{D_{\alpha_2}}\cup\overline{D_{\alpha_3}}$ is injective and its image is the union of $\Delta(\sel|_{\cpack}(\alpha_i))$ and a triangular region $\mathsf{T}=\dev(T)$ bounded by those disks. The disk $\Delta^*$ containing $\mathsf{T}$ as an ideal triangle is contained in $\dev(\mathsf{U})$, so we can define $D^*_\tau=(\dev|_{U})^{-1}(\Delta^*)$.
The set $D^*_\tau$ is a disk in $\widetilde{\Pj}$ formed by $T$ and three lunes, each contained in one of $D_{\alpha_i}$, with vertices at two consecutive tangency points. 
We conclude that those dual disks are pairwise disjoint and $D^*_\tau$ meets only the disks of the circle packing corresponding to the vertices of $\tau$.  

While Proposition \ref{pr: from selection map to circle packing}
establishes that any element in $\sP_\sT$ can be regarded as the selection map of some circle packing on some projective surface, the following proposition implies that the selection map of a pair $(\Pj, \cpack)$ (where of course $\cpack$ is a circle packing over a marked projective surface $\Pj$ with nerve equal to $\sT$) determines the isometry class of $(\Pj, \cpack)$. Combining those two statements we see that $\sP_\sT$ is naturally identified with the moduli space of circle packings over projective surfaces with nerve equal to $\sT$, see Corollary \ref{cor:pt}.

\begin{prop}\label{pr: selection map determines circle packing}
    Two pairs $(\Pj, \cpack)$, $(\Pj',\cpack')$ are isomorphic (through an homotopically trivial isomorphism) if and only if
$[\hol, \sel_{\cpack}]=[\hol', \sel_{\cpack'}]$.
\end{prop}

\begin{proof}
We can pick representatives of the two classes so that $\hol=\hol'$ and $\sel_{\cpack}=\sel_{\cpack'}$.

Let $\dev$ and $\dev'$ be the developing maps with holonomies $\hol$ and $\hol'$. Notice that by the assumption  $\dev(D_\alpha)=\dev'(D'_\alpha)$ for every $\alpha\in\widetilde{\sTk{0}}$, and consequently $\dev(D^*_\tau)=\dev'((D')^*_\tau)$. 
Consider the family of open sets $\{D_\alpha|\alpha\in\widetilde{\sTk{0}}\}\cup\{D^*_\tau|\tau\in\widetilde{\sTk{2}}\}$  of $\widetilde{\Pj}$ 
(resp. the family of opens sets $\{D'_\alpha|\alpha\in\widetilde{\sTk{0}}\}\cup\{(D')^*_\tau|\tau\in\widetilde{\sTk{2}}\}$of $\widetilde{\Pj'}$): they cover the complement $\widetilde{\Pj}^*$ (resp. $\widetilde{\Pj'}^*$) of the tangency points in $\widetilde{\Pj}$ (resp. $\widetilde{\Pj}'$).
Since the nerves of those two families are the same, and the map $\dev$ (resp. $\dev'$) restricts to an injective map on each element, there is a well-defined map
$\tilde F:\widetilde{\Pj}^*\to\widetilde{\Pj'}^*$  such that
$\tilde F|_{D_\alpha}=\left(\dev'|_{D'_\alpha}\right)^{-1}\circ\left(\dev|_{D_\alpha}\right)$ (resp. $\tilde F|_{D^*_\tau}=\left(\dev'|_{(D')^*_\tau}\right)^{-1}\circ\left(\dev|_{D^*_\tau}\right)$).
Clearly $\tilde F$ is bijective, as its inverse is obtained by the same construction reversing  $\Pj$ and $\Pj'$. Moreover, by construction we have $\dev'\circ\tilde F=\dev$. In particular the map $\tilde F$ is holomorphic so it extends to a bijective map (still denoted by $\tilde F$) from $\widetilde{\Pj}$ to $\widetilde{\Pj}$, realizing a projective isomorphism between the universal coverings of $\Pj$ and $\Pj'$. Since the map $\tilde F$ commutes with the action of $\pi_1(\Sigma)$ and sends $D_\alpha$ to $D'_\alpha$, it is easy to check that it descends to a homotopically trivial isomorphism between $\Pj$ and $\Pj'$ sending $\cpack$ to $\cpack'$.
\end{proof}

We can summarize Proposition \ref{pr: from selection map to circle packing} and Proposition \ref{pr: selection map determines circle packing} in this way:

\begin{cor}\label{cor:pt}
    The correspondence $(\Pj,\cpack)\to[\hol,\sel_\cpack]$ induces a bijective map between the moduli space of pairs $(\Pj,\cpack)$  with nerve $\sT$ and the space $\sP_{\sT}$.
\end{cor}

In other words $\sP_{\sT}$ is naturally identified with the moduli space of pairs $(\Pj,\cpack)$, where $\Pj$ is a projective structure on $\Sigma$ and $\cpack$ is a circle packing with nerve $\sT$.

\section{A vanishing theorem}\label{sec:vanishing}
With the definition of the set $\sP_\sT$ defined, we now turn our attention to the proof of the main Theorem \ref{thm:main}. The statements in the main theorem, both manifold structure on $\sP_\sT$ and the projective rigidity, will reduce to a vanishing theorem in our setting.
In this section we describe and prove this technical tool.

\subsection{Statement of the vanishing theorem}

Let us fix $(\rho,\sel)\in\hat\sP_\sT$.
For each edge $\ed\in\widetilde{\sTk{1}}$, let $\mathsf{p}(\ed)$ be the tangent point in $\CP^1$ between the disks corresponding to the endpoints of $\ed$.
Recall from equation \eqref{eq:projvect} that $\vc{\lie{A}}$ is the vector field associated to an element $\lie{A}\in\sl$.

\begin{thm}\label{thm:vanishing}
Let $\lie{P}$ be a $\rho$-equivariant $\sl$-valued $0$-cochain.
If for any oriented edge $\oed$ we have that
\[
\vc{\lie{P}(\oed_+)}(\mathsf{p}(\oed))=\vc{\lie{P}(\oed_-)}(\mathsf{p}(\oed))\,.
\]
then $\lie{P}=0$.
\end{thm}

\subsection{A combinatorial Lemma}
The proof of Theorem \ref{thm:vanishing} relies upon a combinatorial lemma which we state and prove in this subsection.

Let $\sT$ be a quasi-simplicial triangulation of a surface $\Sigma$. 
\begin{defn} \label{defn:decoration}
    A {\it decoration} of  $\sT$  is given by
    \begin{enumerate}
        \item A partition of the set of vertices in two classes: the elements of one of those classes are called white, the elements of the other are called blue,
        \item An orientation of a subset of edges joining two blue vertices.
    \end{enumerate}
\end{defn}

Let us introduce some more language.
 \begin{defn} 
 Assume that a decoration of $\sT$ has been fixed.
 \begin{enumerate}
 
 	\item An edge of $\sT$ is a {\it blue edge} (resp. {\it white edge}) if both of its vertices are blue (resp. white).
 	
 	\item A triangle is a {\it blue triangle} (resp. {\it white} triangle) if all of its vertices are blue (resp. white). 
 
 	 \end{enumerate}
  \end{defn}
Notice that oriented edges are necessarily blue, but some blue edges may not be oriented. 

A {\it corner } of the triangulation is a path $(\ed_1,\alpha,\ed_2)$ in the boundary of a triangle. Following \cite{pak} and \cite{Schramm92:CageEgg},
we assign to each corner a weight, called the {\it inversion} of the corner, that lies in $\{0,1/2,1\}$.
The weight of the corner $(\ed_1,\alpha,\ed_2)$ is $1$ if $\alpha$ is a blue vertex and  either both $\ed_1$ and $\ed_2$ are not oriented, or they are both oriented edges and the vertex $\alpha$ is the initial point of one and the terminus of the other.
The weight is $1/2$ if one edge is oriented and the other is not oriented.
The weight is $0$ in the remaining cases.

\begin{remark}
    Notice that with this definition the inversion of a corner with a white vertex is necessarily $0$.
\end{remark}

The total inversion around a vertex $\alpha$ is defined as the sum of the inversions of corners with vertex at $\alpha$, and will be denoted by $\wh(\alpha)$.

The total inversion of a triangle $\tau$, is similarly defined as the sum of the inversions of corners contained in the triangle.

The following is a simple but key remark, and can be checked by a direct analysis of all cases.

\begin{lemma} \label{lm:inversion-triangle}
If $\tau$ is not a white triangle, 
 the total inversion of $\tau$ is at least $1$.   
\end{lemma}
\begin{proof}
    If no edge in the boundary of $\tau$ is oriented, then since $\tau$ contains at least one blue vertex, we see that the inversion at the corresponding corner is $1$ and we are done.
    
    If there is exactly one oriented edge, then the inversion at corners corresponding to its endpoints are both equal to $1/2$, so the total inversion of the triangle is at least $1$.
    Similarly if there is exactly one non-oriented edge, then the inversion at its endpoints is $1/2$, and we conclude as previously.
    
    Finally, if all edges are oriented, then there is at least one vertex that is neither the initial point nor the terminal point of both of its incident edges. So also in this case the total inversion is at least $1$. 
\end{proof}

As noted, the proof of the vanishing Theorem  \ref{thm:vanishing} for cochains will reduce to a vanishing theorem for decorations, which will roughly say that there are no non-trivial decoration with small uniform total inversions around vertices. 

\begin{defn} \label{defn: tight}
    We say that a decoration is {\it tight} if for any blue vertex $\alpha$ we have $\wh(\alpha)\leq 2$.
\end{defn}

The {\it trivial} decoration is the one for which all vertices are white.
Clearly this decoration is tight in the sense above. 
We may now state our combinatorial vanishing result.
\begin{prop} \label{prop:nobluevertex-bis}
 Let $\Sigma$ be a closed oriented  surface of genus $\genus(\Sigma)\geq 2$. The only tight decoration of any quasi-simplicial triangulation of $\Sigma$ is the trivial one.
\end{prop}
\begin{proof}
    First of all notice that the tightness condition is local. So, the lifting of any tight decorated triangulation to 
    a finite covering of $\Sigma$ is still a tight decorated triangulation.
    This remark and Lemma \ref{lm:quasi-simplicial} permits us to reduce the problem to the case where the triangulation is simplicial.

    Assume now by contradiction that there exists a nontrivial tight decoration on a simplicial triangulation $\sT$ on $\Sigma$.
    Let $U$ be the open subset of $\Sigma$ obtained by removing all the (closed) white triangles; then denote by $U_0$ a connected component of $U$, and by $\Sigma_0$ the closure of $U_0$. The region $\Sigma_0$ is a union of triangles, and its topological boundary is a subgraph of the $1$-skeleton of $\Sigma$ comprising white edges. Notice that  the boundary of $\Sigma_0$ is not in general a union of circles, since it may happen that the valence in $\partial\Sigma_0$ of some vertex is larger than $2$.
    
    In such a case, where $\partial\Sigma_0$ is not a circle, we may  \enquote{desingularize} the boundary of $\Sigma_0$ as follows. Consider the disjoint union of triangles whose interior is in $\Sigma_0$, and glue together the edges that are identified in $\Sigma_0$.

\begin{figure}[htbp]
  \centering
  \includegraphics[width=10cm]{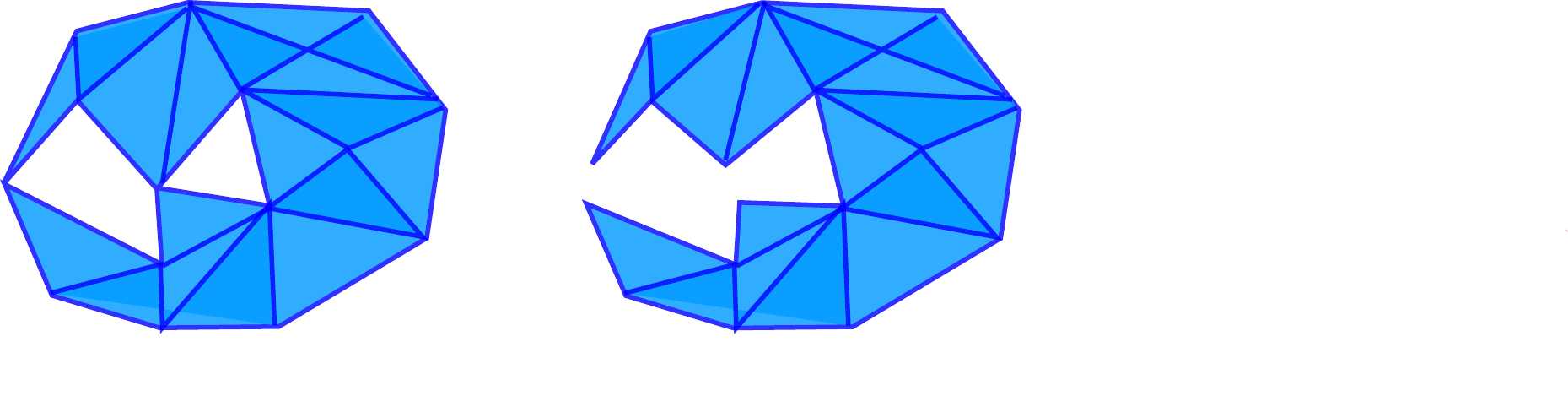}
  \caption{An example of desingularization: on the left $\Sigma_0$, on the right $\hat\Sigma_0$.}
\end{figure}

In this way one obtains $\hat\Sigma_0$, a topological surface with boundary, and the inclusion of all triangles into $\Sigma_0$ factors through the map $\hat\Sigma_0\to\Sigma_0$.
It is easy to see that this map is a homeomorphism in the interior, as well as on open edges of the boundary.
Clearly a natural triangulation is defined on $\hat\Sigma_0$. 
We remark that if $k$ is the number of boundary edges of $\hat\Sigma_0$, then 
\begin{equation}\label{eq:chi}
 \chi(\hat\Sigma_0)-k/2<0\,.   
\end{equation} 
To see this, note that if $U \neq \Sigma$, then the boundary of $\Sigma_0$ is not empty, and so neither is the boundary of $\hat\Sigma_0$.
In this case, then $\chi(\hat\Sigma_0)\leq 1$ and $k\geq 3$. On the other hand, if $U=\Sigma$, then $\hat\Sigma_0=\Sigma_0=\Sigma$ and so $\chi(\hat\Sigma_0)\leq -2$ by the assumption on the topology of the surface $\Sigma$.

Denote by $\wh_0$ the total inversion of $\hat\Sigma_0$, defined as the sum of all inversions of corners contained in $\hat\Sigma_0$.  Since $\hat\Sigma_0$ does not contain white triangles, Lemma \ref{lm:inversion-triangle} implies that $f\leq\wh_0$, where $f$ denotes the number of triangles in $\hat\Sigma_0$.
On the other hand the tightness condition on $\Sigma$, together with the observation that boundary vertices of $\hat\Sigma_0$ are white by construction, implies that $\wh_0\leq 2v_i$, where $v_i$ is the number of interior vertices, so we deduce from the pair of inequalities that  $f\leq 2v_i$. We will prove that this inequality contradicts \eqref{eq:chi}.

Indeed, from Lemma~\ref{lm:comb-triang} we have that $3f=2e_i+k$, where $e_i$ is the number of interior edges of $\hat\Sigma_0$. Since the number of boundary vertices of $\hat\Sigma_0$ equals the number of boundary edges, we see also from Lemma~\ref{lm:comb-triang} that
\[
\chi(\hat\Sigma_0)=v_i-e_i+f=v_i-(3f-k)/2+f=
(2v_i-f+k)/2\geq k/2\,,
\]
       
We then obtain the contradiction with \eqref{eq:chi}, concluding the proof.
\end{proof}

\subsection{Proof of the vanishing Theorem \ref{thm:vanishing}}
We are now in a position to prove Theorem \ref{thm:vanishing}.
\begin{proof}
By the hypothesis  we can construct a family of tangent vectors $\mathsf{v}(\oed)\in T_{\mathsf{p}(\oed)}\CP^1$ for all $\oed\in\widetilde{\sTor}$, such that
\[
   \mathsf{v}(\oed)=\vc{\lie{P}(\alpha)}(\mathsf{p}(\oed))
\]
where $\alpha$ is either of the endpoints of $\oed$. 
The vector $\mathsf{v}(\oed)$ is clearly independent of the orientation of $\oed$, so we can refer to it as a function of the underlying unoriented edge $\ed$, and simply write $\mathsf{v}(\ed)$.

Notice that the invariance of $\lie{P}$ implies that \begin{equation}\label{eq:equiv v}
\mathsf{v}(g\ed)=d(\rho(g))\mathsf{v}(\ed).
\end{equation}

Now, up to multiplying $\lie{P}$ by a phase $e^{i\omega}$ we may moreover assume that $\mathsf{v}(\ed)$ is never tangent to the boundary of the  disks $\Delta(\sel(\oed_\pm))$, whenever it is not zero.

Next we construct a decoration on $\sT^{(0)}$: we will say that $\alpha \in \sT^{(0)}$ is a {\it blue vertex} if $\lie{P}(\tilde{\alpha})\neq 0$ for any lifting $\tilde{\alpha}$ of $\alpha$, while we will say that $\alpha \in \sT^{(0)}$ is a {\it white vertex} if $\lie{P}(\tilde{\alpha}) = 0$.

We fix an orientation on a subset of the set of blue edges according to the following rules. Let $\ed\in\sTk{1}$ be a blue edge.
Take a lifting $\tilde{\ed}$ of $\ed$.

(i) If $\mathsf{v}(\tilde{\ed})=0$, then we do not orient $\ed$.

(ii) If $\mathsf{v}(\tilde{\ed}) \neq 0$, then we orient 
$\ed$ from the endpoint $\alpha$ to the endpoint $\beta$ if and only if $\mathsf{v}(\tilde{\ed})$ outwards from $\partial \Delta(\sel(\tilde{\alpha}))$ and inwards from $\partial\Delta(\sel(\tilde{\beta}))$.
Here of course $\tilde\alpha$ and $\tilde\beta$ refer to the endpoints of $\tilde{\ed}$ (and project to $\alpha$ and $\beta$).

The equivariance \eqref{eq:equiv v} shows that this orientation is well-defined and does not depend on the lifting.

We claim that this decoration is tight in the sense of Definition~\ref{defn: tight}. Indeed, to begin, we observe that for each blue vertex $\alpha$ and for any lifting $\tilde\alpha$,
the vector field $\vc{\lie{P}(\tilde{\alpha})}$ cannot be tangent to $\partial \Delta(\sel(\tilde{\alpha}))$. Otherwise by our nontangency assumption just prior to the decoration just above, the vector field $\vc{\lie{P}(\tilde{\alpha})}$ would vanish at every tangency point of $\partial \Delta(\sel(\tilde{\alpha})$: since there are more than two such tangency points, we would conclude that $\lie{P}(\tilde{\alpha})$ vanishes identically, contrary to the assumption that $\alpha$ is a blue vertex.

Thus, in this case where $\alpha$ is blue, we see that $\partial \Delta(\sel(\tilde{\alpha}))$ divides into two intervals (maybe empty), say $I^+_{\tilde{\alpha}}$ and $I^-_{\tilde{\alpha}}$, defined so that 

(i) $\vc{\lie{P}(\tilde{\alpha})}$ points outwards on $I^+_{\tilde{\alpha}}$, and 

(ii) $\vc{\lie{P}(\tilde{\alpha})}$ points inwards on $I^-_{\tilde{\alpha}}$.

\begin{figure}[htbp]
  \centering
  \includegraphics[width=8cm]{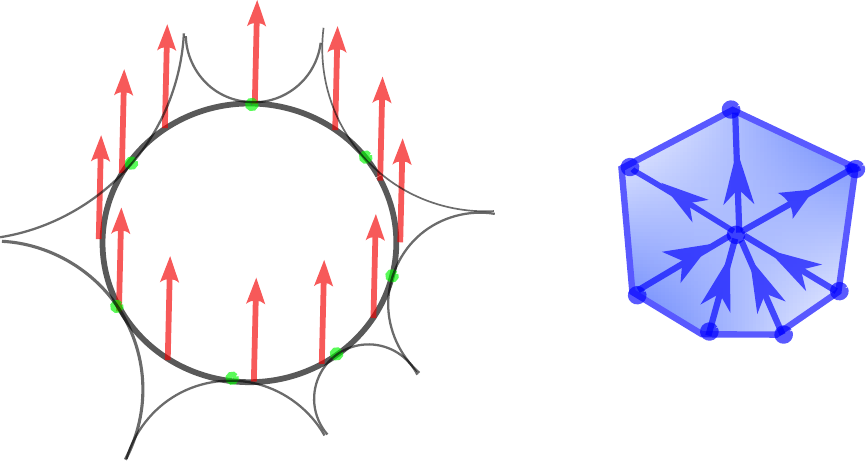}
  \caption{In the example it is shown the case where $\vc{\lie{P}(\tilde{\alpha})}$ is a translational field. The corresponding orientation of the blue edges with endpoint at $\tilde{\alpha}$ is shown on the right.}
\end{figure}

Now, if  an edge $\ed$ with endpoint $\alpha$ is not oriented, then we should have
$0=\mathsf{v}(\tilde\ed)=\vc{\lie{P}(\tilde\alpha)}(\mathsf{p}(\tilde{\ed}))$
, so 
$\mathsf{p}(\tilde{\ed}) \in \overline{I^+_{\tilde{\alpha}}} \cap \overline{I^-_{\tilde{\alpha}}}$ . We conclude that if $\ed$ is not oriented, then $\vc{\lie{P}}(\tilde\alpha)$ vanishes at $\mathsf{p}(\ed)$.  

On the other hand if $(\ed_1,\alpha,\ed_2)$ is a corner with oriented edges and inversion $1$, then $\mathsf{p}(\tilde{\ed_1})$ and $\mathsf{p}(\tilde{\ed_2})$ lie in different open components $I^-_{\tilde{\alpha}}$ and $I^+_{\tilde{\alpha}}$.

We consider the collection $E(\alpha)$ of oriented edges incident to $\alpha$. It follows that only one of the following situations can occur:
 \begin{itemize}
     \item There are two unoriented edges in the set $E(\alpha)$, and corners with oriented edges have all inversion $0$. In this case in the set $E(\alpha)$ we have either one corner of inversion $1$ and two corners with inversion $1/2$ (if the two unoriented edges bound a corner), or four corners of inversion $1/2$. The total inversion around $\alpha$ is $2$.
     \item There is only one unoriented edge the set $E(\alpha)$ of $\alpha$ and corners with oriented edges have all inversion $0$. In this case we have two corners of inversion $1/2$ and the total inversion in $1$.
     \item There is only one unoriented edge in the set $E(\alpha)$ and only one corner with oriented edges and inversion $1$. In this case the total inversion is $2$.
     \item All edges are oriented. In this case, no corner has inversion $1/2$ and the number of corners with inversion $1$ can be $0,1,2$. Again the total inversion around $\alpha$ cannot be bigger than $2$.
 \end{itemize}
 We conclude that the decoration satisfies the condition of tightness of Definition \ref{defn: tight}.

The combinatorial vanishing result, Proposition \ref{prop:nobluevertex-bis}, then shows that all of our vertices are white, and hence $\lie{P}$ vanishes identically.
    
\end{proof}

\section{The space of circle packing with a fixed nerve} \label{sec: manifold theorem}

The aim of this section is to prove the first part of the main Theorem \ref{thm:main} that we summarize in the following statement. 
\begin{thm}
$\sP_{\sT} = \mathbf{p}_{\sM}(\hat{\sP_{\sT}})$ is a submanifold of $\sM_{\sT}$ of dimension $6\genus(\Sigma)-6$.
\end{thm}

In order to prove this statement, we proceed along the following rough outline.

 We need to show that differential of the defining map $\tgmap$ at some point $(\rho,\sel)\in\sP_\sT$ is full rank. Thus, after describing the kernel of the differential in terms of the definition of $\tgmap$ (subsection \ref{ssec:T}), we recognize that kernel in terms of cocycles. In particular for each element in the kernel of the linearization of $\tgmap$, we use the rigidity of complementary triangles to define a quasi-periodic $2$-cochain (subsection \ref{ssec:realization}). We then identify the collection of the co-differentials of those $2$-cochains  with the kernel of the linearization of $\tgmap$.

The next step is to prove that this space of $1$-cocycles is a totally real subspace of $\sC^1_\rho$ (subsection \ref{sec:tot real}). Moreover the complexification of that subspace, named $\sH^{(\rho,\sel)}$, is described in terms of explicit conditions of the behavior of the projective vector fields at the tangency points. With this description, the computation of the dimension of
$\sH^{(\rho,\sel)}$ is amenable to analysis via the vanishing Theorem \ref{thm:vanishing} (subsection \ref{ssec:dimH}).

\subsection{The kernel of the linearization of $\tgmap$.
}\label{ssec:T}
Since $\mathbf{p}_{\sM}$ is a 
submersion whose fibers are six-
dimensional, in order to prove the 
theorem it suffices to check that 
$\mathbf{p}_{\sM}^{-1}(\sP_{\sT})$ 
is a 
submanifold of $\hat{\sM}_{\sT}$ of 
dimension $6\genus(\Sigma)$.

Since $\sP_\sT$ is defined as the intersection of $\tgmap^{-1}(0)$ and the open subset $\sM_*$,
our first goal becomes proving that elements of $\sP_\sT$ are regular points  of the map $\tgmap$, or, similarly that elements in $\hat\sP_\sT$ are regular points of the map $\hat\tgmap$. In that case, we would conclude from Lemma~\ref{lem:M-hat_T} that $\hat{\sP}_{\sT}$ is a submanifold of $\hat{\sM}_{\sT}$ of dimension (now simplifying the notation by setting $\genus(\Sigma) = \genus)$),
\begin{align*}
    12\genus - 6 + 3v -e &= 12\genus - 6 - 6\genus + 6\\
    &= 6\genus
\end{align*}
We notice that we have once again used the Euler characteristic computation $v-e+f = 2-2\genus$ together with the surface triangulation relation $3f=2e$ to conclude that $3v-e=6-6\genus$.

Now fix $(\rho, \sel) \in \hat{\sP}_{\sT}$.  We wish to prove that the kernel $T_{(\rho, \sel)}$ of $d_{(\rho, \sel)} \hat\tgmap$ has dimension $6\genus$: here we derive from the defining conditions \ref{defn:I-hat} that this space $T_{(\rho, \sel )}$ is defined by

\begin{align}\label{defn:T-rho-f}
    T_{(\rho, \sel)} =&\{(\coc, \dot{\sel}) \in \Zgroup \times (\herm)^{\tilde{\sT}_0}:\\\notag & \dot{\sel}(\alpha)\in T_{\sel(\alpha)}\dS \textrm{ for all } \alpha\in \widetilde{\sTk{0}} ,\\\notag
    &\dot{\sel}(g \alpha) = \rho(g)\cdot\dot{\sel}(\alpha) + \Yc{\coc(g)}(\sel(g\alpha)) \textrm{ for all } \alpha\in \widetilde{\sTk{0}}, g\in\pi_1(\Sigma), \\\notag
    &\mink{\sel(\oed_-), \dot{\sel}(\oed_+)} + \mink{\sel(\oed_+), \dot{\sel}(\oed_-)}=0 \textrm{ for all } \oed\in\widetilde{\sTor}\}\notag.
\end{align}

Note that the first two conditions summarize the fact that $(\coc,\dot\sel)$ is an element of $T_{(\rho,\sel)}\hat{\sM}_{\sT}$. The last condition is obtained by linearizing $\hat\tgmap$ at $(\rho,\sel)$.

We remark, following the discussion of \eqref{eq:primitive equivariance}, that the space $T_{(\rho, \sel)}$ contains the tangent space of the $\SL(2,\C)$-orbit through $(\rho, \sel)$, i.e.
\begin{align}\label{defn:B-rho-f}
   B_{(\rho, \sel)} =\{&(\coc, \dot{\sel}) \in \Zgroup \times (\herm)^{\tilde{\sT}_0}: \\ \notag &\qquad\text{ there is } \coc_0 \in \fsl(2,\C) \text{ so that}\\ \notag &\coc(g) = \coc_0 - \Ad \rho(g)\coc_0,\ \text{for all }\ g\in\pi_1(\Sigma)\ \textrm{and }\\ \notag
   & \dot{\sel}(\tilde{\alpha})) = \Yc{\coc_0}(\sel(\alpha)) \text{ for all } \alpha\in\sTk{0}\}.
\end{align}
It is easy to see that
$\dim B_{(\rho, \sel)}= 6$. So we must prove that 
\begin{equation}\label{eq:goal1}
\dim T_{(\rho, \sel)}/B_{(\rho, \sel)} = 6\genus-6\,.
\end{equation}

\begin{remark}\label{rk:tangent space}
    Once we prove that result, we will be able to conclude $T_{[\rho, \sel]}\sP_{\sT} = T_{(\rho, \sel)}/B_{(\rho, \sel)}$.
\end{remark}

\subsection{A realization lemma.}\label{ssec:realization} We start with a basic lemma on promoting infinitesimal information on edges to infinitesimal data on triangles.

\begin{lemma} \label{lem:defnPfromthreevectors}
Let us fix $(\coc, \dot{\sel}) \in  T_{(\rho, \sel)}$.
For every triangle $\tau \in \widetilde{\sTk{2}}$  there is a unique infinitesimal motion $\lie{R}(\tau) \in \fsl(2,\C)$ so that 
\begin{align*}
    \Yc{\lie{R}(\tau)}(\sel(\alpha)) &= \dot{\sel}(\alpha)\,,\\
    \Yc{\lie{R}(\tau)}(\sel(\beta)) &= \dot{\sel}(\beta)\,,\\
    \Yc{\lie{R}(\tau)}(\sel(\gamma)) &= \dot{\sel}(\gamma)\,,
\end{align*}
where $\alpha,\beta,\gamma$ are the vertices of $\tau$.

Moreover the map $\lie{R}:\widetilde{\sTk{2}}\to\sl$ given by $\tau\mapsto\lie{R}(\tau)$  is quasi-periodic with period equal to $\coc$.
    That is, it satisfies the relation

    \begin{equation}\label{eq:rqp}
      \lie{R}(g\cdot\tau) = \Ad \rho(g)\cdot \lie{R}(\tau) + \coc(g). 
    \end{equation}
    
\end{lemma}
\begin{proof}
We have both that $\sel(\alpha), \sel(\beta)$ and $\sel(\gamma)$ are linearly independent in $\herm$ and moreover, that those vectors span a subspace of signature $(2,1)$: to see these assertions, note that the Gram  matrix representing the scalar product with respect to these vectors is 
\begin{equation*}
\begin{pmatrix}
1 & -1 & -1\\
-1 & 1 & -1 \\
-1 & -1 & 1\\
\end{pmatrix}
\end{equation*}
which is both non-degenerate and has signature $(2,1)$.

Next, let $\mathsf{E}\in\herm$ be a vector orthogonal to each of $\sel(\alpha), \sel(\beta)$ and $\sel(\gamma)$, and then construct a vector $\mathsf{E'} \in \herm$ out of the new basis $\mathsf{E}$, $\sel(\alpha), \sel(\beta)$ and $\sel(\gamma)$ by
\begin{align}\label{eq:defV}
    \mink{\mathsf{E'}, \sel(\alpha)} &= - \mink{\mathsf{E}, \dot{\sel}(\alpha)}\,,\\
    \mink{\mathsf{E'}, \sel(\beta)}&= - \mink{\mathsf{E}, \dot{\sel}(\beta)}\,,\nonumber \\
    \mink{\mathsf{E'}, \sel(\gamma)} &= - \mink{\mathsf{E}, \dot{\sel}(\gamma)}\,,\nonumber \\
    \mink{\mathsf{E'}, \mathsf{E}} &=0\nonumber\,.
\end{align}
Then we define $\mathscr{L} \in \End(\herm)$ by
\begin{align*}
    \mathscr{L}(\sel(\alpha)) &= \dot{\sel}(\alpha)\,,\\
    \mathscr{L}(\sel(\beta)) &= \dot{\sel}(\beta)\,,\\
   \mathscr{L}(\sel(\gamma)) &= \dot{\sel}(\gamma)\,,\\
    \mathscr{L}(\mathsf{E}) &= \mathsf{E'}\,.
\end{align*}

Now, because $(\coc, \dot{\sel}) \in  T_{(\rho, \sel)}$, by our choice of $\mathsf{E'}$, we see that the transformation $\mathscr{L}$ is skew-symmetric. Since the correspondence $\lie{A}\to \Yc{\lie{A}}$ realizes an isomorphism between $\sl$ and the space $\mathfrak{so}(\mink{\cdot,\cdot})$ of skew-symmetric linear maps of $\herm$, there exists $\lie{R}(\tau)\in\sl$ such that $\mathscr{L}=\Yc{\lie{R}(\tau)}$\,.

Moreover, suppose we have another element $\lie{R}'$ of $\sl$ which satisfies the given conditions. Then, using that $\Yc{\lie{R}'}$ is skew-symmetric, we see that the vector $\Yc{\lie{R}'}(\mathsf{E})$ satisfies the conditions \eqref{eq:defV}, and so coincides with $\mathsf{E'}$. So $\Yc{\lie{R}'}$  would act on the basis $\mathsf{E}$, $\sel(\alpha), \sel(\beta)$ and $\sel(\gamma)$ exactly as does $\Yc{\lie{R}(\tau)}$.  We conclude $\lie{R}'=\lie{R}(\tau)$. 

Finally in order to prove that $\lie{R}$ satisfies \eqref{eq:rqp}, for a fixed $g\in\pi_1(\Sigma)$ we define
    \begin{equation*}
        \lie{R_0}(\tau) := \Ad \rho(g)\cdot \lie{R}(\tau) + \coc(g),
    \end{equation*} 
    then we prove that 
     \begin{align}
        \Yc{\lie{R_0}}: \sel(g\alpha) &\mapsto \dot{\sel}(g\alpha)\label{eq: R under translation}\\
        \Yc{\lie{R_0}}: \sel(g\beta) &\mapsto \dot{\sel}(g\beta)\notag\\
         \Yc{\lie{R_0}}: \sel(g\gamma) &\mapsto \dot{\sel}(g\gamma)\,,\notag\\
         \notag
    \end{align}
     for all edges $\alpha,\beta,\gamma$ of $\tau$: this will then imply that $\lie{R_0}(\tau)=\lie{R}(g\cdot\tau)$, and hence that the element $\lie{R}$ defined by the lemma satisfies the translation law \eqref{eq:rqp}.

     To see \eqref{eq: R under translation}, we consider a vertex, say $\alpha$, and recall from  Definition~\ref{defn:T-rho-f} that 
     \begin{equation*}
         \dot{\sel}(g \alpha) = \rho(g)\cdot\dot{\sel}(\alpha) + \Yc{\coc(g)}(\sel(g\alpha))\,.
     \end{equation*}
     Then

\begin{align*}
         \dot{\sel}(g\alpha) 
  &= \rho(g)\cdot\dot{\sel}(\alpha)+\Yc{\coc}[\sel(g\alpha)] \\
  &= \rho(g)\cdot\Yc{\lie{R}(\tau)}\sel(\alpha)+ \Yc{\coc}[\sel(g\alpha)]  \text{ by the defining relation \ref{lem:defnPfromthreevectors} for } \lie{R}(\tau) \\
 &= \rho(g) \cdot \Yc{\lie{R}(\tau)}[\rho(g))^{-1} \cdot \sel(g \alpha)]+ \Yc{\coc}[\sel(g\alpha)]\\
  &= \Yc{\Ad(\rho(g))\lie{R}(\tau)}(\sel(g(\alpha))+\Yc{\coc}[\sel(g\alpha)]\text{ using the intertwining of } \Ad\\
 &\qquad\qquad\qquad\qquad\qquad\qquad\text{and } d\Phi, \text{ i.e. } \Ad(\Phi(h))\cdot \Yc{X} = \Yc{\Ad(h)X}\\
&= \Yc{\lie{R}_0}\sel(g(\alpha)),
\end{align*}
    as desired.
\end{proof}

\begin{remark}\label{rem: geometrical meaning of R}
Let us give a more geometric interpretation of $\lie{R}(\tau)$.
Let us assume that $(\coc,\dot\sel)$ is in fact the tangent vector of a family $(\rho_t,\sel_t)$ in $\hat\sP_\sT$ passing through $(\rho,\sel)$ at $t=0$.
(Notice that so far this is a priori a restrictive hypothesis, since we have not still proved that $\hat\sP_\sT$ is a submanifold in $\hat\sM_\sT$.)
Consider the triangular region $\mathsf{T}_t$ bounded by $\Delta_t(\alpha):=\Delta(\sel_t(\alpha))$, $\Delta_t(\beta):=\Delta(\sel_t(\beta))$, $\Delta_t(\gamma):=\Delta(\sel_t(\gamma))\,.$

Since a triangular region is determined by its vertices, we see there exists a smooth family of projective transformations $R(t)$ sending $\mathsf{T}_0$ to $\mathsf{T}_t$, such that $R(0)=\mathrm{Id}$.

Notice that $R(t)(\Delta_0(\alpha))=\Delta_t(\alpha)$, and so $R(t)\cdot\sel(\alpha)=\sel_t(\alpha)$. 
Analogously, we have $R(t)\cdot\sel(\beta)=\sel_t(\beta)$, and $R(t)\cdot\sel(\gamma)=\sel_t(\gamma)$.

From that we deduce that $\frac{dR}{dt}(0)=\lie{R}(\tau)$. Indeed clearly the transformation $\Yc{\frac{dR}{dt}(0)}$ sends $\sel(\alpha)$ to $\dot\sel(\alpha)$, $\sel(\beta)$ to $\dot\sel(\beta)$ and $\sel(\gamma)$ to $\dot\sel(\gamma)$.

In other words we can observe that each dual disk corresponding to $\sel_t^*(\tau)$ contains but three tangency points, and $R(t)$ is the unique projective transformation sending $\sel^*(\tau)$ to $\sel^*_t(\tau)$, so $\lie{R}(t)$ measures the infinitesimal displacement of the disk $\Delta(\sel^*(\tau))$ which is consistent with the motion of the tangency points.
\end{remark}

Returning to the discussion in Lemma \ref{lem:defnPfromthreevectors}, observe that by the formula \eqref{eq:rqp}, the form  $\lie{Q} = \delta(\lie{R})$ is a $\rho$-equivariant discrete $1$-cocycle with $\per_\delta(\lie{Q})=\coc$, where we recall  $\per_\delta$ from  \eqref{eq:delta-period map}.
Recall that for an oriented edge $\oed$, we denote by $\tau(\oed)$ denotes the unique triangle containing $\oed$ and inducing the right orientation. Then we have
\[
   \lie{Q}(\oed)=\lie{R}(\tau(\oed))-\lie{R}(\tau(\roed))
\]
on oriented edges $\oed$ as follows. 

 From the discussion in subsection~\ref{ssec: chain complexes} we have that $\lie{Q}$ is $\rho$-equivariant and $\delta$-closed.

In this way we have described a map
\[
   \mathscr{Q}:T_{(\rho,\sel)}\to Z^1(\sC^*_\rho,\delta)\subset \sC^1_\rho
\]
defined by $\mathscr{Q}(\coc,\dot{\sel})=\lie{Q}=\delta \lie{R}$. 

\begin{lemma}
$\ker\mathscr{Q}=B_{(\rho,\sel)}$.
\end{lemma}
\begin{proof}
    Notice that if $(\coc, \dot{\sel}) \in B_{(\rho, \sel)}$, then by definition (see \eqref{defn:B-rho-f}) there is some $\coc_0$ so that $\lie{R}(\tau)=\coc_0$ for every choice of $\tau\in\sTk{2}$ -- and hence we  conclude $\lie{Q}$ vanishes identically. 
    
    Conversely, if $\lie{Q}=\mathscr{Q}(\coc,\dot{\sel})$ vanishes identically then $\lie{R}$ is constant, so there exists $\coc_0$ such that $\lie{R}(\tau)=\coc_0$ for all $\tau\in\sTk{2}$. This immediately implies from the definition of $\lie{R}$ that for all $\alpha\in\widetilde{\sTk{0}}$ we have $\dot{\sel}(\alpha)=\Yc{\lie{R}(\tau)}(\alpha)=\Yc{\coc_0}(\sel(\alpha))$
    where $\tau$ is a triangle with a vertex at $\alpha$.
    Moreover it is simple to check that $\lie{R}$ is quasi-periodic with period equal to $\coc(g)=\coc_0-\Ad\rho(g)\coc_0$.  
   Those observations show  that $(\coc, \dot{\sel}) \in B_{(\rho, \sel)}$.
\end{proof}

From this Lemma we see, after changing perspectives from considering $T_{(\rho, \sel)}$ to considering $\mathscr{Q}$, that \eqref{eq:goal1} is equivalent to

\begin{equation}\label{eq:goal2}
    \dim\Ima\,\mathscr{Q}=6\genus-6\,
\end{equation}
for all $(\rho,\sel)\in\hat{\sP}_{\sT}$.
Notice, moreover, that from this perspective, $\Ima\,\mathscr{Q}$ is naturally identified to the tangent space of $\sP_{\sT}$ at the point $[\rho, \sel]$.

\subsection{A totally real subspace.}\label{sec:tot real}
This section is devoted to proving equation \eqref{eq:goal2}.
In the first part we will prove that the image of the map $\mathscr{Q}$ is a totally real subspace in $\sC^1_\rho$, giving a precise characterization of it.
In the second part, using our description of the complexification of $\Ima\mathscr{Q}$, we compute its (complex) dimension via cohomological techniques and the vanishing theorem.

We begin by defining two useful spaces. 
\begin{defn}
    Set
    \begin{align*}
        \sH^{(\rho,\sel)}_{\R} =
       \{\lie{Q} \in \sC^1_{\rho}: \delta\lie{Q} =0&,
         \text{ and } 
          \vc{\lie{Q}(\oed)} \text{ is tangent to both } \partial \Delta(\sel(\oed_-)) \text{ and } \partial \Delta(\sel(\oed_+)\}\\
         &\text{for all }\oed\in\widetilde{\sTor}\}\,,\\
        \sH^{(\rho,\sel)}=\{\lie{Q} \in \sC^1_{\rho}:\delta\lie{Q} =0&, \vc{\lie{Q}(\oed)} \text{  has a double zero at } \mathsf{p}(\oed)\}\,,
    \end{align*}
    where $\mathsf{p}(\oed)\in\CP^1$ is the tangency point between the disks $\Delta(\sel(\oed_-))$ and $\Delta(\sel(\oed_+))$, and coincides (as a line in $\C^2$) with the kernel of $\sel(\oed_-)+\sel(\oed_+)$.
\end{defn}

\begin{prop} \label{prop:totally real subspace}
(i) $\Ima\,\mathscr{Q}=\sH^{(\rho,\sel)}_\R$.

(ii) $\sH^{(\rho,\sel)}_{\R}$ is a totally real subspace of $\sH^{(\rho,\sel)}$.
\end{prop}
\begin{proof}
    First let us prove that $\lie{Q}=\mathscr{Q}(\coc,\dot{\sel})$ lies in $\sH^{(\rho,\sel)}_\R$.
    Indeed, if $\lie{R}\in\tilde\sC^2$ is constructed as in Lemma \ref{lem:defnPfromthreevectors}, we have that 
    \[
    \lie{Q}(\oed)=\delta\lie{R}(\oed) =\lie{R}(\tau(\oed))-\lie{R}(\tau(\roed))\,.
    \]
    Notice that the endpoints $\oed_-$ and $\oed_+$ are vertices of both $\tau(\oed_-)$ and $\tau(\oed_+))$, so by the definition of $\lie{R}$ we have
    \begin{align*}
    \Yc{\lie{R}(\tau(\oed))}(\sel(\oed_-))&=
     \Yc{\lie{R}(\tau(\roed))}(\sel(\oed_-))=
     \dot{\sel}(\oed_-)\\
     \Yc{\lie{R}(\tau(\oed))}(\sel(\oed_+))&=
     \Yc{\lie{R}(\tau(\roed))}(\sel(\oed_+))=
     \dot{\sel}(\oed_+)
    \end{align*}
    and from this we conclude that 
    \[
     \Yc{\lie{Q}(\oed)}(\sel(\oed_-))=
     \Yc{\lie{Q}(\oed)}(\sel(\oed_+))=0.
    \]
    By one implication of Lemma \ref{lm: stabilizer of a disk} we conclude that the vector field $\vc{\lie Q(\oed)}$
    is tangent to both $\partial\Delta(\sel(\oed_+))$ and $\partial\Delta(\sel(\oed_-))$\,.
    
    Conversely, begin with an element $\lie{Q} \in \sH^{(\rho,\sel)}_{\R}$. By the last paragraph of Subsection \ref{ssec: chain complexes} we see that there is an infinitesimal translation $\lie{R}: \widetilde{\sTk{2}} \to \fsl(2,\C)$ of triangles so that 
    \begin{align*}
         \lie{Q}(\oed) &= \lie{R}(\tau(\oed)) - \lie{R}(\tau(\roed))\\
    \end{align*}
    for all $\oed\in\widetilde{\sTor}$.
    Recall $\lie{R}$ is quasi-periodic of period $\coc\in \Zgroup$.

    Now, note that since $\lie{Q} \in \sH^{(\rho,\sel)}_{\R}$, we also have that $\vc{\lie{Q}(\oed)}$ is tangent to $\partial \Delta(\sel(\oed_\pm))$,  and so applying the other implication of Lemma \ref{lm: stabilizer of a disk} we have that $\Yc{\lie{Q}(\oed)}(\sel(\oed_\pm)) = 0$. 
   
   We deduce 
   \begin{equation*}
\Yc{\lie{R}(\tau(\oed))}(\sel(\oed_\pm))= \Yc{\lie{R}(\tau(\roed))}(\sel(\oed_\pm)) \text{ for all } \oed\in\widetilde{\sTor}.
   \end{equation*}
We conclude that if $\tau$ and $\tau'$ are two adjacent triangles, the maps $\Yc{\lie{R}(\tau)}$ and $\Yc{\lie{R}(\tau')}$ agree at $\sel(\alpha)$ where $\alpha$ is a common vertex.
   
   Now, since the link of $\alpha$ is connected, 
   the vector $\Yc{\lie{R}(\tau)}(\sel(\alpha))$ is independent of $\tau$, for any $\tau\in\widetilde{\sTk{2}}$ with a vertex at $\alpha$.
     
     There thus is a deformation $\dot{\sel}(\alpha)$ so that
   \begin{equation*}
       \dot{\sel}(\alpha) = \Yc{\lie{R}(\tau)}(\sel (\alpha))
   \end{equation*}
   for every triangle $\tau$ with a vertex at $\alpha$.

    It is then immediate to see that $(\coc, \dot{\sel}) \in T_{(\rho, \sel)}$ from the definition of $T_{(\rho, \sel)}$ and the skew-symmetry of $\Yc{\lie{R(\tau)}}$.
    
    Notice that by construction the element $\lie{R}(\tau)$ satisfies conditions in  Lemma~\ref{lem:defnPfromthreevectors}.
    So by definition $\mathscr{Q}(\coc, \dot{\sel}) = \delta\lie{R}=\lie{Q}$.

   To see the second part of the proposition, we note that given any two disks $\Delta_1$ and $\Delta_2$ tangent at a vertex $\mathsf{p}$, any infinitesimal motion $\lie{A} \in \fsl(2,\C)$ with a double zero at $\mathsf{p}$ may be uniquely decomposed into a sum $\lie{A}=\lie{A}_1 + i\lie{A}_2$, where  $\vc{\lie{A}_1}$ and $\vc{\lie{A}_2}$ are both tangent to $\partial \Delta_1$  and $\partial \Delta_2$.  Here notice that $\vc{i\lie{A}_2}$ is normal to both $\partial \Delta_1$  and $\partial \Delta_2$.

   Now given $\lie{Q} \in \sH^{(\rho,\sel)}$, we consider for an arbitrary $\oed \in \widetilde{\sTor}$, the decomposition $\lie{Q}(\oed) = \lie{Q}_1(\oed) + i \lie{Q}_2(\oed)$, with respect to the disks $\Delta(\sel(\oed_-))$ and $\Delta(\sel(\oed_+))$ tangent at $\mathsf{p}(\oed)$.   
 It is then easy to see that each $\lie{Q}_i \in \sC^2_{\rho}$, and that $\lie{Q}_i(\oed)$ corresponds to a vector field tangent to  both $\partial \Delta(\sel(\oed_-))$  and $\partial \Delta(\sel(\oed_+))$.

 We claim that $\delta \lie{Q}_1 = 0$ and $\delta \lie{Q}_2 = 0$.  Indeed, of course since $\delta \lie{Q} =0$, we have that 
 $0= (\delta \lie{Q})(\alpha)= (\delta \lie{Q}_1)(\alpha) + i (\delta \lie{Q}_2)(\alpha)$ for all $\alpha\in\widetilde{\sTk{0}}$. We consider the nature of each term in the sum. Of course, the vector field corresponding to the element  $$(\delta \lie{Q}_1)(\alpha) = \sum_{\oed\,:\,\oed_-=\alpha} \lie{Q}_1(\oed)$$ is a sum of vector fields tangent to $\partial \Delta(\sel(\alpha))$, while the vector field corresponding to $i(\delta \lie{Q}_2)(\alpha)$ is a sum of vector fields normal to $\partial \Delta(\sel(\alpha))$. We then conclude from the earlier relation $(\delta \lie{Q}_1)(\alpha) + i(\delta \lie{Q}_2)(\alpha) =0$ that each term must vanish, proving the claim. 
 
 We then conclude that $\lie{Q}_1, \lie{Q}_2 \in \sH_{\R}^{(f,\rho)}$.
 
 In this way we have proved that $\sH^{(\rho,\sel)}=\sH^{(\rho,\sel)}_{\R}+ i \sH^{(\rho,\sel)}_{\R}$. In order to conclude that $\sH^{(\rho,\sel)}_{\R}$ is totally real subspace it is then sufficient to prove that
 $\sH^{(\rho,\sel)}_{\R}\cap i\sH^{(\rho,\sel)}_{\R}=\{0\}$.
 Again if there are $\lie{Q}_1,\lie{Q}_2\in\sH^{(\rho,\sel)}_{\R}$ such that $\lie{Q}_1=i\lie{Q}_2$, for any edge $\oed\in\widetilde{\sTor}$ the vector field corresponding to $\lie{Q}_1(\oed)$ should be both tangent and normal to $\partial \Delta(\sel(\oed_-))$, so it should vanish. We conclude then that $\lie{Q}_1=\lie{Q}_2=0$.
 \end{proof}
 
\subsection{The dimension of $\sH^{(\rho,\sel)}$ }\label{ssec:dimH}
 By Proposition \ref{prop:totally real subspace} in order to prove that $\dim_{\R}\sH^{(\rho,\sel)}_{\R}=6\genus-6$, it is sufficient to prove that $\dim_{\C} \sH^{(\rho,\sel)} = 6\genus-6$.

 To that end, set
 \begin{equation*}
     \sV^{(\rho,\sel)}= \{\lie{Q} \in \sC^1_{\rho}: \vc{\lie{Q}(\oed)} \text{ has a double zero at } \mathsf{p}(\oed)
\}, \end{equation*}
where we recall that $\mathsf{p}(\oed)$ is the tangency point between $\Delta(\sel(\oed_-))$ and $\Delta(\sel(\oed_+))$.

Of course, our definition of $\sH^{(\rho,\sel)}$ then provides that 
\begin{equation*}
\sH^{(\rho,\sel)} = \ker \delta \cap \sV^{(\rho,\sel)} = \ker \delta\bigr\vert_{\sV^{(\rho,\sel)}}.
\end{equation*}

 Notice that $\sV^{(\rho,\sel)}$ has complex dimension $\dim \sV^{(\rho,\sel)} = e$, as the value of a vector field at the point of tangency may be described by a constant vector field $a \frac{\partial}{\partial z}$, when we put the point of tangency at infinity in the complex plane. Recall the space $\sC^0_{\rho}$ of equivariant functions on $\widetilde{\sTk{0}}$ with image in $\fsl(2,\C)$; since $\dim_{\C} \sC^0_{\rho} = 3v$ and since by Lemma \ref{lm:comb-triang}, we have $e-3v = 6\genus-6$, we see that proving the claim that $\dim_{\C} \sH^{(\rho,\sel)} = 6g-6$ is equivalent to showing that 
 \begin{equation*}
     \delta\big\vert_{\sV^{(\rho,\sel)}} : \sV^{(\rho,\sel)} \to \sC^0_{\rho}
 \end{equation*}
 is surjective.

 This last statement is itself then equivalent to showing that for every non-vanishing functional $\psi \in (\sC^0_{\rho})^{\ast}\setminus\{0\}$, we have that $\psi \circ \delta$ does not vanish identically on $\sV^{(\rho,\sel)}$.

 We next examine this last statement, using the non-degenerate pairing on $\sC^0_{\rho}$. From that perspective, our statement in the previous paragraph is equivalent by duality to showing that for any non-zero $\lie{P}\in\sC^0_\rho$ the functional

 \begin{align*}
     \phi_{\lie{P}}: \sC^1_{\rho} &\to \C\\
     \lie{Q} &\mapsto \pair{\lie{P}, \delta \lie{Q}} 
 \end{align*}
 does not vanish on $\sV^{(\rho,\sel)}$.

Of course, by Proposition \ref{pr:adjoint} we have that 
\begin{equation*}
    \phi_{\lie{P}}(\lie{Q}) = \pair{\lie{P}, \delta \lie{Q}} = - \pair{d\lie{P}, \lie{Q}}
\end{equation*}
so we are reduced to proving the following proposition.

\begin{prop}\label{pr:no exact ortogonal to V}
    Let $\lie{P}\in\sC^0_\rho$.
    If $\pair{d\lie{P}, \lie{Q}} =0$ for every $\lie{Q} \in \sV^{(\rho,\sel)}$, then $d\lie{P} =0$.  (Indeed, $\lie{P}=0$.)
\end{prop}

\begin{proof}
Fix an oriented edge $\oed_0\in\widetilde{\sTor}$, and  a non-zero element $\lie{A}\in\sl$ with a double zero at $\mathsf{p}(\oed_0)$ and define $\lie{Q}\in\sC^1_\rho$ by
\begin{align*}
\lie{Q}(g\cdot\oed_0)&=\Ad(\rho(g))\lie{A}\\
\lie{Q}(g\cdot\roed_0)&=-\Ad(\rho(g))\lie{A} \quad\\
\lie{Q}(\oed)&=0 \quad\mathrm{otherwise}\,.
\end{align*}
Now it is immediate to check that $\lie{Q}\in \sV^{(\rho,\sel)}$.
On the other hand, $\pair{d \lie{P}, \lie{Q}}=\killi{\lie{P}((\oed_0)_{+}) - \lie{P}((\oed_0)_-) , \lie{A}}$, so assuming that $\lie{P}$ satisfies the hypothesis, we deduce that $\killi{\lie{P}((\oed_0)_+) - \lie{P}((\oed_0)_-) , \lie{A}} = 0$. From Remark \ref{rk:eval} we deduce that the vector fields $\vc{\lie{P}((\oed_0)_+)}$ and $\vc{\lie{P}((\oed_0)_-)}$ agree at the tangency point $\mathsf{p}(\oed_0)$:
\[
\vc{\lie{P}((\oed_0)_+)}(\mathsf{p}(\oed_0))=\vc{\lie{P}((\oed_0)_-)}(\mathsf{p}(\oed_0))\,.
\]
Since the edge $\oed_0$ has been chosen arbitrarily, we see that the hypothesis of the vanishing Theorem \ref{thm:vanishing} is satisfied, so we conclude $\lie{P}=0$.

 \end{proof}

\section{Projective local rigidity} \label{sec: projective rigidity}

In this section we consider the  map $\Hol$
which associates to any marked projective structure with circle packing its holonomy.
Using the description of the moduli space of those structures given in Section \ref{sec:fixed nerve}, we simply have

\begin{align*}
        \Hol: \sP_{\sT} &\to \chi\\
        [(\rho, \sel)] &\mapsto [\rho]
    \end{align*}

The aim of this section is to prove Theorem~\ref{thm:main}(ii), i.e. that the holonomy locally determines the structure. This statement is equivalent to a local rigidity of the circle packing on a given projective structure: we cannot deform the circle packing without deforming the underlying projective structure.

\begin{thm}\label{th:proj rigidity}
    $\Hol:\sP_{\sT} \to \chi$ is an immersion.
\end{thm}

With the notation as in Definition\eqref{defn:T-rho-f}and Equation \eqref{eq:goal1},  the tangent space of $\sP_\sT$ at a point $[(\rho,\sel]\in\sP_\sT$ can be identified with the quotient  $T_{(\rho,\sel)}\sP_{\T}=T_{(\rho,\sel)}/B_{(\rho,\sel)}$ (compare Remark \ref{rk:tangent space}).
From this perspective the differential of $\Hol$ is the map
 \begin{align*}
        d \Hol: T_{[(\rho,\sel)]}\sP_{\sT} &\to T_\rho\chi= H^1(\pi_1, \fsl(2,\C))\\
        [(\coc, \dot{\sel})] &\mapsto [\coc]\,.
    \end{align*}

Let us denote by $T^0$ the subspace of $T_{(\rho,\sel)}$ containing elements of the form $(0,\dot{\sel)}$.

\begin{lemma}
 Any element of $\ker d\Hol$ admits a representative in $T^0$.
\end{lemma}
\begin{proof}
Let $[(\coc,\dot{\sel})]\in\ker d\Hol$.
Then there exists $\coc_0\in\sl$ such that $\coc(g)=\coc_0-\Ad\rho(g)\coc_0$ for all $g\in\pi_1(\Sigma)$.
Define $\mathsf{Y}:\widetilde{\sTk{0}}\to \herm$  by $\mathsf{Y}(\alpha)=\Yc{\coc_0}(\sel(\alpha))$.
Notice that $(\coc,\mathsf{Y})\in B_{(\rho,\sel)}$,  so $(0,\dot{\sel}-\mathsf{Y})$ is a different representative of
$[(\coc,\dot{\sel})]$ lying in $T^{0}$.
\end{proof}

Theorem \ref{th:proj rigidity} is proven by combining the Lemma above with the following Proposition.

\begin{prop}\label{pr:proj rigidity}
    $T^0=\{0\}$.
\end{prop}

In order to prove Proposition \ref{pr:proj rigidity} 
we observe that for every $\mathsf{X}\in\dS$, the Lie algebra of $\Stab_{\PSL(2,\C)}(\mathsf{X})$, say
$\mathfrak{st}(\mathsf{X})$, is a totally real subspace of $\sl$ conjugated to $\mathfrak{sl}(2,\R)$. We already know by Lemma \ref{lm: stabilizer of a disk} that elements in $\mathfrak{st}(\mathsf{X})$ correspond to projective vector fields tangent to $\partial\Delta(\mathsf{X})$, and we deduce that elements of the subspace $i\,\mathfrak{st}(\mathsf{X})$ correspond to projective vector fields normal to $\partial\Delta(\mathsf{X})$.

Clearly if $\mathsf{X}\in\dS$ is fixed the evaluation map
\[
 \sl\to T_{\mathsf{X}}\dS \quad \lie{A}\mapsto\Yc{\lie{A}}(\mathsf{X}) 
\]
factors through an isomorphism
\[
   i\,\mathfrak{st}(\mathsf{X})\to T_{\mathsf{X}}\dS
\]
and we deduce the following
\begin{lemma} \label{lem: A(x, dot X)}
   Given $\mathsf{X}\in\dS$ and $\dot{\mathsf{X}}\in T_{\mathsf{X}}\dS$, there exists a unique $\lie{A}=\lie{A}(\mathsf{X},\dot{\mathsf{X}})\in\sl$ such that
   \begin{itemize}
       \item $\Yc{\lie{A}}(\mathsf{X})=\dot{\mathsf{X}}$;
       \item $\vc{\lie{A}}$ is normal to $\partial\Delta(\mathsf{X})$.
   \end{itemize}
   Moreover if $\lie{B}\in\sl$ satisfies $\Yc{\lie{B}}=\dot{\mathsf{X}}$, then on $\partial\Delta(\mathsf{X})$ the vector field $\vc{\lie{A}}$ coincides with the normal component of $\vc{\lie{B}}$.
\end{lemma}

Geometrically the restriction of the vector field $\vc{\lie{A}}$ to the boundary of $\Delta(\mathsf{X})$ encodes the infinitesimal displacement of the disk  $\Delta(\mathsf{X})$, when $\mathsf{X}$ is moved in the direction $\dot{\mathsf{X}}$.

More precisely, if $\mathsf{X}(t)$ is a 
path in $\dS$ passing at $t=0$ through 
$\mathsf{X}$ with velocity 
$\dot{\mathsf{X}}$, we can choose a path 
$B(t)\in\PSL(2,\C)$ such that $B(t)\mathsf{X}=\mathsf{X}(t)$, with 
$B(0)=\mathrm{Id}$.
Notice that $B(t)$ is determined up to pre-composition by elements in the stabilizer of $\mathsf{X}$, so we deduce that its derivative $\lie{B}=\frac{dB}{dt}(0)$ is determined  up to summing elements in $\mathfrak{st}(\mathsf{X})$.
In particular it is determined by the obvious condition $\Yc{\lie{B}}(\mathsf{X})=\dot{\mathsf{X}}$.
So we see that on $\partial\Delta(\mathsf{X})$  the vector field $\vc{\lie{A}}$ on  is the normal component of $\vc{\lie{B}}$.
That is, for any displacement of disks $\Delta(\mathsf{X}(t))=B(t)\Delta(\mathsf{X})$, we have that for any point  $\mathsf{p}\in\partial\Delta(\mathsf{X})$ $\vc{\lie{A}}(\mathsf{p})$ is the normal component of the vector $\frac{d B(t)\mathsf{p}}{dt}(0)=\vc{\lie{B}}(\mathsf{p})$.

We can now prove Proposition \ref{pr:proj rigidity}.

\begin{proof}
     Let $(0,\dot\sel)\in T^0$.
     Let us consider the following $0$-cocycle $\lie{P}$: for each $\alpha\in\widetilde{\sTk{0}}$ put
     $\lie{P}(\alpha):=\lie{A}(\sel(\alpha), \dot\sel(\alpha))$.
    Since both $\sel$ and $\dot\sel$ are $\rho$-equivariant, we conclude that $\lie{P}$ is $\rho$-equivariant too, so it lies in $\sC^0_\rho$.

    Let us prove that $\lie{P}$ satisfies the hypothesis of the vanishing Theorem \ref{thm:vanishing}.  Fix a path $(\rho_t,\sel_t)\in\hat\sP_\sT$ passing through $(\rho,\sel)$ with velocity equal to $(0,\dot\sel)$.
    Notice that for every edge $\oed$ the disks corresponding to $\sel_+(t):=\sel_t(\oed_+)$ and $\sel_-(t):=\sel_t(\oed_-)$ are tangent for all $t$.
    
     In particular there is a smooth path $B(t)$ in $\PSL(2,\C)$ such that
    $B(t)\cdot\sel(\oed_-)=\sel_-(t)$ and $B(t)\cdot\sel(\oed_+)=\sel_+(t)$.

    Putting $\lie{B}=\frac{dB(t)}{dt}(0)$
    we have from Lemma~\ref{lem: A(x, dot X)} that on $\partial\Delta(\sel(\oed_-))$ the vector field $\vc{\lie{P}(\oed_-)}$
    equals the normal component of $\vc{\lie{B}}$, while  the vector field $\vc{\lie{P}(\oed_+)}$ coincides with the normal component of $\vc{\lie{B}}$ on $\partial\Delta(\sel(\oed_+)) $.

    Since those circles meet tangentially at $\mathsf{p}(\ed)$, we conclude that both $\vc{\lie{P}(\oed_-)}$ and 
    $\vc{\lie{P}(\oed_+)}$ coincide at 
    $\mathsf{p}(\ed)$ with the normal component of $\vc{\lie{B}}$.

    So, we are once again in the setting of the vanishing Theorem \ref{thm:vanishing}. Thus we may conclude that $\lie{P}=0$. By definition of $\lie{P}=\lie{A}(\sel,\dot\sel)$, the vanishing of $\lie{P}$ implies the vanishing of $\dot\sel$.
     
\end{proof}
    
\section{KMT immersion}\label{sec: KMT} 
In the seminal paper \cite{Kojima-Mizushima-Tan:Projective} the description of circle packings over complex projective surfaces is based on a quantity,  already introduced in \cite{HeSchramm}, which parameterizes all circle packings whose nerve is a union of two adjacent triangles, up to the  $\PSL(2,\C)$ action.
More precisely in \cite{Kojima-Mizushima-Tan:Projective}
an element $(\rho, \sel)\in\hat\sP_\sT$ is associated with a function 
$\mathbf{cr}_{(\rho,\sel)}:\sTk{1}\to\R$, called the cross ratio function of the circle packing, and defined in the following way.

Given an edge $\ed\in\widetilde{\sTk{1}}$ with endpoints $\alpha$ and $\beta$, as usual denote by $\mathsf{p}(\ed)$ the tangency point of the disks $\sel(\alpha)$ and $\sel(\beta)$.
Now let $\tau_1,\tau_2$ be the two triangles adiacent to $\ed$, and denote by $\ed_1,\ed_2$  $\ed_3$,$\ed_4$ the edges in the boundary of $\tau_1\cup\tau_2$ . Choose the enumeration coherently with the orientation of the boundary, and so that $\ed_1$ and $\ed_2$ lie in $\tau_1$.
Let us set
\[
\widetilde{\mathbf{cr}}_{(\rho,\sel)}(\ed)=i[\mathsf{p}(\ed):\mathsf{p}(\ed_1):\mathsf{p}(\ed_2):\mathsf{p}(\ed_4)]\,.
\]
\begin{figure}[htbp]
  \centering
  
\includegraphics[width=12cm]{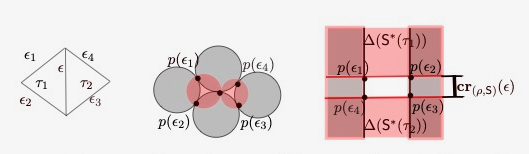}
  \caption{The number $\mathbf{cr}_{(\rho,\sel)}(\ed)$ corresponds to the distance between the disks $\Delta(\sel^*(\tau_1))$, and $\Delta(\sel^*(\tau_2))$ in the chart where $\mathsf{p}(\ed)=\infty, \mathsf{p}(\ed_1)=0,\mathsf{p}(\ed_2)=1$. }\label{fig:eccentricity}
\end{figure}

In \cite{HeSchramm}  it is proven that this function is always real-valued, and for each edge $\ed$ determines the configuration of the $4$ disks corresponding to vertices contained in $\tau_1\cup\tau_2$.
Indeed let $\sel^*:\widetilde{\sTk{2}}\to\dS$
be the map corresponding to the dual configuration. Fix coordinates on $\CP^1$ so that $\mathsf{p}(\ed)=\infty, \mathsf{p}(\ed_1)=0,\mathsf{p}(\ed_2)=1$.
Then $\Delta(\sel^*(\tau_1))$ is the standard upper half-plane, while
$\Delta(\sel^*(\tau_2))$ is a lower half-plane.
It is easy to see that the number $\widetilde{\mathbf{cr}}_{(\rho,\sel)}(\ed)$ corresponds to the distance (measured in that chart) between those half-planes (see Figure \ref{fig:eccentricity}).

The function is clearly invariant under the action of $\pi_1(\Sigma)$ on $\widetilde{\sTk{1}}$ and so descends to a function $\mathbf{cr}_{(\rho,\sel)}:\sTk{1}\to\R$.

In \cite{Kojima-Mizushima-Tan:Projective} the map
\[
\mathbf{cr}:\sP_{\sT}\to \R^{\sTk{1}}\quad [\rho,\sel]\mapsto \mathbf{cr}_{(\rho,\sel)}
\]
is shown to be injective and the image is proven to be a semi-algebraic set.

In this section we prove that the map $\mathbf{cr}$ is in fact an immersion.

Recall for $[\rho,\sel]\in\sP_\sT$ that $T_{[\rho,\sel]}\sP_{\sT}$ has been identified with $\sH^{(\rho,\sel)}_{\R}$ in Section \ref{sec:tot real}, so the differential of $\mathbf{cr}$ can be regarded as
\[
d_{[\rho,\sel]}\mathbf{cr}:\sH^{(\rho,\sel)}_{\R}\to\R^{\sTk{1}}\,.
\]

\begin{lemma}
Let $\lie{Q}\in\sH^{(\rho,\sel)}_{\R}$ and set $\dot{\mathbf{cr}}=d_{[\rho,\sel]}\mathbf{cr}(\lie{Q})$.
For every $\oed\in\widetilde{\sTor}$ we have that  $\dot{\mathbf{cr}}(\ed)=0$ if and only if $\lie{Q}(\tilde{\oed})=0$, where $\tilde{\oed}$ is any lifting of $\oed$.
\end{lemma}
\begin{proof}
    Let $[\rho_t,\sel_t]$ be a path in $\sP_{\sT}$ through $[\rho,\sel]$ whose velocity at $t=0$ corresponds to $\lie{Q}$.
    Recall that this means that if $\Delta^*_t(\tau)$ is the disk corresponding to $\sel^*(\tau)$, and $R_t(\tau)\in\PSL(2,\C)$ is the transformation sending $\Delta^*(\tau)$ to $\Delta^*_t(\tau)$, in a way that tangent points are sent to tangent points, then $\lie{Q}=\delta\lie{R}$, where $\lie{R}(\tau)=\frac{d R_t(\tau)}{dt}(0)$ (see Remark \ref{rem: geometrical meaning of R}).

    In other words, $\lie{Q}(\oed)$ measures the relative infinitesimal  displacement of the disk $\Delta^*(\tau(\oed))$ from $\Delta^*(\tau(\roed))$ along a deformation $(\rho_t,\sel_t)$.

    Recall now that $\lie{Q}(\oed)$ is a projective vector field with double zero at $\mathsf{p}(\ed)$ and that it is tangent to the disks corresponding to the endpoints of $\ed$. 
    If we fix coordinates as above, so that $\Delta^*(\tau(\roed))$ and $\Delta^*(\tau(\oed))$ appear respectively as a upper half-plane and a lower half-plane, like in Figure \ref{fig:eccentricity},  in that chart $\vc{\lie{Q}(\oed)}$ is  a translational vector field parallel to the imaginary axis.

    In particular its norm measures (up to the sign) the infinitesimal variation of $\dot{\mathbf{cr}}$.
    So the latter is zero if and only if $\lie{Q}(\ed)=0$.
\end{proof}

\bibliographystyle{alpha}

\bibliography{liter10-19}

\end{document}